\documentclass[reqno]{amsart}
\usepackage{amsmath}
\usepackage{mathrsfs}
\usepackage{amssymb}
\usepackage{amsthm}
\usepackage{amsfonts}
\usepackage{mathtools}
\usepackage{bm}
\usepackage[normalem]{ulem}  
\usepackage{fancyhdr}
\usepackage{hyperref}
\usepackage{framed}
\usepackage{float}
\usepackage[pdftex,final]{graphicx}
\usepackage{tikz}
\usepackage{comment}
\usepackage{sidecap}
\usepackage{xcolor}
\usepackage{pgfplots}
\usetikzlibrary{calc,arrows}
\usetikzlibrary{patterns}

\usepackage[iso-8859-7]{inputenc}

\theoremstyle{plain}
\newtheorem{thm}{Theorem}[section]
\newtheorem{corollary}[thm]{Corollary}
\newtheorem{lemma}[thm]{Lemma}
\newtheorem{prop}[thm]{Proposition}
\newtheorem*{theorem*}{Theorem}
\theoremstyle{definition}
\newtheorem{rem}[thm]{Remark}
\newtheorem{dfn}[thm]{Definition}

\numberwithin{equation}{section}

\newcommand{\mytilde}{\raise.17ex\hbox{$\scriptstyle\mathtt{\sim}$}}

\newcommand{\I}{\ensuremath{\mathcal{I}}}
\newcommand{\bfI}{\ensuremath{\boldsymbol{\mathcal{I}}}}
\renewcommand{\S}{\ensuremath{\mathcal{S}}}
\renewcommand{\P}{\ensuremath{\mathcal{P}}}
\newcommand{\R}{\ensuremath{\mathcal{R}}}

\newcommand{\N}{\ensuremath{\mathcal{N}}}

\newcommand{\G}{\ensuremath{\mathcal{G}}}

\newcommand{\U}{\ensuremath{\mathcal{U}}}
\newcommand{\E}{\ensuremath{\mathcal{E}}}

\newcommand{\RgeO}{\ensuremath{\mathbb{R}_{ \geq 0}}}
\newcommand{\Rat}[1]{\ensuremath{\mathbb{R}^{#1}}}
\newcommand{\mbf}[1]{\ensuremath{\boldsymbol{#1}}} 
\renewcommand{\bf}[1]{\textbf{#1}}

\begin{document}

\author{D. Boskos}
\address{Department of Automatic Control, School of Electrical Engineering, KTH Royal Institute of Technology, Osquldas v\"ag 10, 10044, Stockholm, Sweden}
\email{boskos@kth.se}

\author{D. V. Dimarogonas}
\address{Department of Automatic Control, School of Electrical Engineering, KTH Royal Institute of Technology, Osquldas v\"ag 10, 10044, Stockholm, Sweden}
\email{dimos@kth.se}

\begin{abstract}
In this report, we aim at the development of an online abstraction framework for multi-agent systems under coupled constraints. The motion capabilities of each agent are abstracted through a finite state transition system in order to capture reachability properties of the coupled multi-agent system over a finite time horizon in a decentralized manner. In the first part of this work, we define online abstractions by discretizing an overapproximation of the agents' reachable sets over the horizon. Then, sufficient conditions relating the discretization and the agent's dynamics properties are provided, in order to quantify the transition possibilities of each agent.  
\end{abstract}


\title{Online Abstractions for Interconnected Multi-Agent Control Systems}

\maketitle

\section{Introduction}

During the last decade there has been an emerging focus on the problem of high level planning for  multi-agent systems by leveraging methods from formal verification \cite{LsKk04}. In order to  
exploit these tools for dynamic agents, it is required to build a discretized model of the continuous system which allows for the algorithmic synthesis of high level plans. Specifically, the use of an appropriate abstract representation enables the conversion of discrete paths into sequences of feedback controllers which enable the continuous time model to implement the high level specifications.

The aforementioned control synthesis problem has lead to a significant  research effort for the derivation of discrete state analogues of  continuous control systems, also called abstractions, which can capture reachability properties of the original model. Abstractions for piecewise affine systems on simplices and rectangles were introduced in  \cite{HlVj01} and have been further studied in \cite{BmGm14}. Closer related to the control framework that we adopt here for the derivation of the discrete models is the paper \cite{HmCp14b} which builds on the notion of In-Block Controllability \cite{CpWy95}. Abstractions for nonlinear systems include \cite{Rg11}, which is focused on general discrete time systems and \cite{AaTaSa09}, where box abstractions are derived for polynomial and other classes of systems. Furthermore, abstractions for interconnected systems have been recently developed in \cite{TyIj08}, \cite{PgPpDm14},  \cite{PgPpDm16}, \cite{RmZm15}, \cite{Mp15a}, \cite{DeTp15} and are primarily based on small gain criteria. 

In this work we consider multi-agent systems and provide an online abstraction methodology which enables the exploitation of the system's dynamic properties over bounded reachable sets. Specifically, we focus on agents whose dynamics consist of decentralized feedback interconnection terms and additional bounded input terms which allow for the synthesis of high level plans under the coupled constraints. The analysis builds on parts of the framework introduced in our recent work \cite{BdDd15a}, which focused on the discretization of the whole workspace and required the assumption of global bounds for the dynamics of the agents. In this framework, the latter hypothesis is considerably weakened, since it is only required that the system is forward complete. In addition, it is also possible to obtain coarser discretizations, since (i) the transition system of each agent is updated at the end of the time interval and thus, heterogeneous discretizations are considered for different agents, and  (ii) the dynamics bounds of each agent, which constitute a measure of ``coarseness" for its discretization, are evaluated for overapproximations of the agent and its neighbors' reachable sets and can result in reduced size discrete models for agents with weaker couplings over the time horizon. A relevant abstraction approach can be also found in \cite{SeAa13} where local Lipschitz properties of probability densities for stochastic kernels are exploited for the efficient abstraction of probabilistic systems into finite Markov Chains.   
  
The rest of the paper  is organized as follows. Basic notation and preliminaries are introduced in Section 2. In Section 3, we formulate well posed online abstractions for single integrator multi-agent systems, based on the existence of appropriate hybrid feedback laws and prove that the latter provide solutions consistent with the design requirement on the systems' free inputs over the specified time horizon. Section 4 is devoted to the design of the specific hybrid controllers that are exploited for the derivation of the transitions. Space-time discretizations that guarantee well posed abstractions and their reachability properties are quantified in Section 5 and we conclude in Section 6. 

\section{Preliminaries and Notation}

We use the notation $|x|$ for the Euclidean norm of a vector $x\in\Rat{n}$. For a subset $S$ of $\Rat{n}$, we denote by  ${\rm int}(S)$ its interior and define the distance from a point $x\in\Rat{n}$ to $S$ as $d(x,S):=\inf\{|x-y|:y\in S\}$. Given $R>0$ and $y\in\Rat{n}$, we denote by $B(R)$ the closed ball with center $0\in\Rat{n}$ and radius $R$, namely $B(R):=\{x\in\Rat{n}:|x|\le R \}$ and $B(x;R):=\{y\in\Rat{n}:|x-y|\le R \}$. Given two sets $A,B\subset\Rat{n}$ their Minkowski sum is defined as $A+B:=\{x+y\in\Rat{n}:x\in A, y\in B\}$. We say that a continuous function $a:\RgeO\to\RgeO$ belongs to class $\mathcal{K}_+$ if it is positive and strictly increasing and that $\beta:\RgeO\times\RgeO \to\RgeO$ is of class $\mathcal{K}_+\mathcal{K}_+$, if $\beta(t,\cdot)$ and $\beta(\cdot,s)$ are of class $\mathcal{K}_+$ for all $t,s\ge 0$.  

Consider a multi-agent system with $N$ agents. For each agent $i\in\N:=\{1,\ldots,N\}$ we use the notation $\mathcal{N}_i\subset\N\setminus\{i\}$ for the set of its neighbors and $N_i$ for its cardinality. We also consider an ordering of the agent's neighbors which is denoted by $j_1,\ldots,j_{N_i}$ and define the $N_i$-tuple $j(i)=(j_1(i),\ldots,j_{N_i}(i))$. Whenever it is clear from the context, the argument $i$ will be omitted from the latter notation. The agents' network is represented by a directed graph $\G:=(\N,\E)$, with vertex set $\N$ the agents' index set and edge set $\E$ the ordered pairs $(\ell,i)$ with $i,\ell\in\N$ and $\ell\in\N_i$. The sequence $i_0i_1\cdots i_m$ with $(i_{\kappa-1},i_{\kappa})\in\E$, $\kappa=1,\ldots,m$, namely, consisting of $m$ consecutive edges in $\G$, forms \textit{a path} (of length $m$) in $\G$. A path $i_0i_1\cdots i_m$ with $i_0=i_m$ is called a \textit{cycle}. Given the nonempty indexed sets $\I_1,\ldots,\I_N$, their Cartesian product $\bfI:=\I_1\times\cdots\times\I_N$ and an agent $i\in\N$ with neighbors $j_1,\ldots,j_{N_i}\in\N$, we define the mapping ${\rm pr}_i:\bfI\to\bfI_i:=\I_i\times\I_{j_1}\times\cdots\times\I_{j_{N_i}}$ which assigns to each $N$-tuple $(l_1,\ldots,l_N)\in\bfI$ the $N_i+1$-tuple $(l_i,l_{j_1},\ldots,l_{j_{N_i}})\in\bfI_i$, i.e., the indices of agent $i$ and its neighbors.

We proceed by providing the definition of a transition system.

\begin{dfn}
A transition system is a tuple $TS:=(Q,Q_0,Act,\longrightarrow)$, where:

\textbullet\; $Q$ is a set of states.

\textbullet\; $Q_0\subset Q$ is a set of initial states.

\textbullet\; $Act$ is a set of actions.

\textbullet\; $\longrightarrow$ is a transition relation with $\longrightarrow\subset Q\times Act\times Q$.

\noindent The transition system is said to be finite, if $Q$ and $Act$ are finite sets. We also denote an element $(q,a,q')\in\longrightarrow$ as  $q\overset{a}{\longrightarrow} q'$ and define ${\rm Post}(q;a):=\{q'\in Q:(q,a,q')\in\longrightarrow\}$, for every $q\in Q$ and $a\in Act$.
\end{dfn}

\section{Abstraction of the Agents Reach Sets}

We focus on multi-agent systems with single integrator dynamics
\begin{equation}\label{single:integrator}
\dot{x}_{i}=f_{i}(x_{i},\bf{x}_j)+v_{i}, i\in\N,
\end{equation}

\noindent with $\bf{x}_j(=\bf{x}_{j(i)}):=(x_{j_{1}},\ldots,x_{j_{N_i}})\in\Rat{N_i n}$. We assume that the agents are in general heterogeneous and consider  decentralized control laws consisting of two terms, a feedback term $f_{i}(\cdot)$ which depends on the states of $i$ and its neighbors, and a free input $v_{i}$. We assume that for each $i\in\N$ it holds $x_{i}\in \Rat{n}$ and that each $f_{i}(\cdot)$ is locally Lipschitz. We also assume that $v_{i}\in\mathcal{U}_{i}$, $i\in\N$ where $\mathcal{U}_{i}$ is a bounded subset of $L^{\infty}(\RgeO;\Rat{n})$ taking values in a compact set $U_i\subset\Rat{n}$ for each $i$ and define $\mathcal{U}:=\mathcal{U}_{1}\times\cdots\times\mathcal{U}_{N}$. The online abstraction framework is based on the discretization of each agent's reachable set over a given time horizon and the selection of a time step $\delta t$ which corresponds to the duration of the discrete transitions. We will consider specific types of space discretizations, called cell decompositions (see also \cite{Gl02}).

\begin{dfn} \label{cell:decomposition}
\textit{(i)} Let $D$ be a bounded domain of $\Rat{n}$. A cell decomposition $\mathcal{S}=\{S_{l}\}_{l\in\mathcal{I}}$ of $D$, is a finite family of bounded sets $S_{l}$, $l\in\mathcal{I}$ with nonempty interior, such that ${\rm int}(S_{l})\cap {\rm int}(S_{\hat{l}})=\emptyset$ for all $l\ne\hat{l}$ and $\cup_{l\in\mathcal{I}} S_{l}=D$. 

\noindent \textit{(ii)} Given a bounded domain $D$ of $\Rat{n}$, a cell decomposition $\S$ of $D$ and a set $A\subset D$, we say that $\S$ is compliant with $A$, if for any $S\in\S$ with $S\cap A\ne \emptyset$ it holds that $S\subset A$.
\end{dfn}

In order to provide decentralized abstractions we follow parts of the approach employed in \cite{BdDd15a} and design appropriate hybrid feedback laws in place of the $v_{i}$'s in order to guarantee well posed transitions. Before proceeding to the necessary definitions related to the problem formulation, we provide certain assumptions on the dynamics of the multi-agent system.
We assume that system \eqref{single:integrator} is \textit{forward complete}, i.e., that for every initial condition $x_0\in\Rat{Nn}$ and input $v\in\mathcal{U}$ the solution $x(t,x_0;v)$ is defined for all $t\ge 0$. Hence, there exists a function $\beta\in\mathcal{K}_+\mathcal{K}_+$ (see \cite{Ki05}) such that 
\begin{equation} \label{beta:bound}
|x(t,x_0;v)|\le\beta (t,|x_0|),\forall t\ge 0, x_0\in\Rat{Nn},v\in\mathcal{U}.
\end{equation}

\noindent Additionally, we assume that each free input $v_{i}$, $i\in\N$ is bounded by a positive constant $v_{\max}(i)$ and in particular, that
\begin{equation}\label{input:bound}
\U_i=\{v_i\in L^{\infty}(\RgeO;\Rat{n}):|v_{i}(t)|\le v_{\max}(i),\forall t\ge 0\}.
\end{equation}

\begin{rem}
The same analysis can be applied in the case where $\mathcal{U}_i$ is the set of all measurable functions $v_i:\RgeO\to U_i$ with $U_i$ a subset of $\Rat{n}$ with nonempty interior. In this case, we can select $v_{i0}\in U_i$ and $R_i>0$ such that $B(v_{i0};R_i)\subset U_i$ and perform the analysis with the feedback terms $\bar{f}_i(x_i,\bf{x}_j):=f_i(x_i,\bf{x}_j)+v_{i0}$ instead, and $v_{\max}(i):=R_i$.
\end{rem}

\noindent We next also provide certain basic properties of the deterministic control system \eqref{single:integrator} which can be found in \cite[Chapter 1]{KiJj11}, or \cite[Chapter 2]{Se98} and will be invoked later in the proofs. For each $r>0$ we define the shift operator ${\rm Sh}_r:\mathcal{U}\to\mathcal{U}$ as $
{\rm Sh}_r(v)(t):=v(t+r),\forall t\ge 0$. To the control system \eqref{single:integrator} we associate the transition map $\varphi:A_{\varphi}\to \Rat{Nn}$ with $A_{\varphi}:=\{(t,t_0,x_0;v):t\ge t_0\ge 0,x_0\in \Rat{Nn},v\in\mathcal{U}\}$, where $\varphi(t,t_0,x_0;v)$ denotes the value at time $t$ of the unique solution of \eqref{single:integrator} with initial condition $x_0$ at time $t_0$ and input $v(\cdot)$. Notice, that by virtue of the forward completeness assumption, $\varphi(\cdot)$ is well defined. The map $\varphi(\cdot)$ satisfies the following properties:

\noindent \textbullet\;\textit{(Strict) Causality.}\; For each $t> t_0\ge 0$, $x_0\in \Rat{Nn}$ and $v^1,v^2\in\mathcal{U}$ with $v^1|_{[t_0,t)}=v^2|_{[t_0,t)}$ it holds $\varphi(t,t_0,x_0;v^1)=\varphi(t,t_0,x_0;v^2)$, where $v^1|_{[t_0,t)}$ denotes the restriction of $v^1(\cdot)$ to $[t_0,t)$.

\noindent \textbullet\;\textit{Semigroup Property.}\; For each $t_2\ge t_1\ge t_0\ge 0$, $x_0\in \Rat{Nn}$ and $v\in\mathcal{U}$ it holds $
\varphi(t_2,t_1,\varphi(t_1,t_0,$ $x_0;v);v)=\varphi(t_2,t_0,x_0;v)$.

\noindent \textbullet\;\textit{Time Invariance.}\; For each $r>0$, $t\ge t_0\ge r$, $x_0\in \Rat{Nn}$ and $v\in\mathcal{U}$ it holds $\varphi(t,t_0,x_0;v)=\varphi(t-r,t_0-r,x_0;{\rm Sh}_r(v))$.

In order to employ the online abstraction framework, we will consider a fixed time horizon $[0,T]$, $T>0$ on which we aim to abstract the agents' dynamics through a finite state transition system.  Thus, at time $t=0$, given the agents' initial positions, we will discretize an overapproximation of their reachable set over $[0,T]$ and select a time step $\delta t$ which exactly divides  $T$, i.e., such that $T=\ell\delta t$ for certain $\ell\in\mathbb{N}$, in order to capture the motion of the system over that time interval through a finite transition system. After employing a discrete plan over $[0,T]$, we repeat the same procedure for the positions of the agents at $t=T$ and the new horizon $[T,2T]$, and proceed analogously with the horizons $[\kappa T,(\kappa+1) T]$, $\kappa\ge 2$. For the subsequent analysis, we will assume fixed the initial states $X_{10},\ldots,X_{N0}$ of all agents at the beginning of the horizon $[0,T]$ and consider for each agent $i\in\N$ an \textit{open overapproximation} $\R_i(t)$ of its reachable set at $t\ge 0$. We also define the union of the reachable sets $\R_i(t)$ over a time interval $[a,b]\subset[0,\infty)$  as $\R_i([a,b]):=\cup_{t\in[a,b]}\R_i(t)$ and their inflation by a certain constant $c>0$ as $\R_i^c(t):=\R_i(t)+B(c)$, $\R_i^c([a,b]):= \cup_{t\in[a,b]}\R_i^c(t)$, implying that
\begin{align} 
\R_i^{c+\bar{c}}([a,b]) =\R_i^c([a,b])+B(\bar{c}),\forall c,\bar{c}>0. \label{Rti:over:interval:inflaed:semigroup}
\end{align}

\noindent By the forward completeness assumption we may always assume that the reachable sets $\R_i([a,b])$ are bounded, since from \eqref{beta:bound}, it follows that ${\rm int}(B(\beta(b+1,|\bf{X}_0|)))$, $\bf{X}_0=(X_{10},\ldots,X_{N0})$ is always an open overapproximation for the reachable set of each $i$ over $[a,b]$.  We can thus obtain by continuity of the feedback terms $f_{i}(\cdot)$ the following bounds for each agent $i\in\N$ over the overapproximations of the reachable sets. In particular we pick constants $M(i)>0$ such that
\begin{equation} \label{dynamics:bound}
|f_{i}(x_{i},\bf{x}_j)|\le M(i), \forall x_{i}\in \R_i([0,T]),x_{\kappa}\in \R_{\kappa}([0,T]),\kappa\in\N_i.
\end{equation}

\noindent Apart from the time horizon $[0,T]$ we will consider for certain technical reasons an additional time duration $0<\tau<T$ which corresponds to an upper bound on the time discretization step $\delta t$. Based on this time duration and the previously derived bounds we consider for each $i\in\N$ the sets
$\R_i^{c_i(\tau)}([0,T-\tau])$, where 
\begin{equation} \label{constant:ci}
c_i(\sigma):=(M(i)+v_{\max}(i))\sigma,\sigma>0,
\end{equation}

\noindent with $M(i)$ and $v_{\max}(i)$ as given in \eqref{dynamics:bound} and \eqref{input:bound}, respectively. We also assume that without any loss of generality it holds
\begin{equation} \label{reachset:includedin:sigma:inflated}
\R_i^{c_i(\sigma)}([0,T-\tau])\supset \R_i(T-\tau+\sigma),\forall \sigma\in(0,\tau]
\end{equation}

\noindent (otherwise we can replace $\R_i(T-\tau+\sigma)$ by $\R_i^{c_i(\sigma)}([0,T-\tau])\cap\R_i(T-\tau+\sigma)$, which by \eqref{input:bound}, \eqref{dynamics:bound} and \eqref{constant:ci} is a again an open overapproximation satisfying \eqref{dynamics:bound}). Hence, it follows from 
\eqref{reachset:includedin:sigma:inflated} that 
\begin{equation} \label{reachset:includedin:sigma:interval:inflated}
\R_i^{c_i(\sigma)}([0,T-\tau])\supset \R_i([0,T-\tau+\sigma]),\forall \sigma\in(0,\tau].
\end{equation}

\noindent Given a time step $0<\delta t<\tau$ we  depict the overapproximations of the reachable sets $\R_i([0,T-\tau])\subset \R_i([0,T-\delta t])\subset\R_i([0,T])$ of agent $i$ with the red areas in Fig. \ref{fig:reach}. They all contain the exact reachable set $\R_i^{\rm exact}([0,T-\tau])$ of $i$ over $[0,T-\tau]$ and the initial condition $X_{i0}$ of $i$. We also depict the inflation  $\R_i^{c_i(\tau)}([0,T-\tau])$ of $\R_i([0,T-\tau])$ which contains the overapproximation $\R_i([0,T])$ as required by \eqref{reachset:includedin:sigma:interval:inflated}. The same property is also illustrated for the set $\R_i^{c_i(\tau-\delta t)}([0,T-\tau])$ enclosed in the dashed curve, which satisfies \eqref{reachset:includedin:sigma:interval:inflated} with $\sigma=\tau-\delta t$.

\begin{figure}[H]
\begin{center}

\resizebox{0.8\textwidth}{!}{
\begin{tikzpicture}[scale=.8]


\filldraw[line width=.02cm, fill=cyan!30!white] plot [domain=90:270, variable=\theta,samples=50]({1.7*cos(\theta)},{1.7*sin(\theta)})
-- (5,-3.5)  plot [domain=-90:90, variable=\theta,samples=50]({5+3.5*cos(\theta)},{3.5*sin(\theta)}) -- (0,1.7); 

\filldraw[line width=.02cm, fill=red!40!white, opacity=.7] plot [domain=90:270, variable=\theta,samples=50]({.2*cos(\theta)},{.2*sin(\theta)})
-- (5,-2)  plot [domain=-90:90, variable=\theta,samples=50]({5+2*cos(\theta)},{2*sin(\theta)}) -- (0,.2); 

\filldraw[line width=.02cm, fill=red!20!white]plot  [domain=90:-90, variable=\theta,samples=50]({5+2*cos(\theta)},{2*sin(\theta)})
-- (5,-2)  plot [domain=-90:90, variable=\theta,samples=50]({5+3.4*cos(\theta)},{2*sin(\theta)}); 

\filldraw[line width=.04cm, fill=red!30!white]plot  [domain=90:-90, variable=\theta,samples=50]({5+2*cos(\theta)},{2*sin(\theta)})
-- (5,-2)  plot [domain=-90:90, variable=\theta,samples=50]({5+2.6*cos(\theta)},{2*sin(\theta)}); 

\draw[line width=.02cm, dashed, cyan] plot [domain=90:270, variable=\theta,samples=50]({.8*cos(\theta)},{.8*sin(\theta)})
-- (5,-2.7)  plot [domain=-90:90, variable=\theta,samples=50]({5+2.7*cos(\theta)},{2.7*sin(\theta)}) -- (0,.8);

\filldraw[line width=.02cm, fill=green!45!white, opacity=.7]  (0,0) to[out=10,in=180] (4.5,1.2) to[out=0,in=90] (6,1.5) to[out=-90,in=90] (5.5,-1) to[out=-90,in=0] (3,-.5) to[out=180,in=0] (0,0); 

\filldraw[]  (0,0) circle (0.05cm);
\node (init) at (0,0) [label=left:$X_{i0}$] {};

\node (init) at (3,0) [label=center:$\R_i^{\rm exact}({[0,T-\tau]})$] {};

\filldraw[]  (6.5,1) circle (0.05cm);
\draw[line width=.03cm,->,>=stealth] (6.5,1)  -- (9.5,1);

\node (init) at (9.5,1) [label=right:$\R_i({[0,T-\tau]})$] {};

\filldraw[]  (7.2,.2) circle (0.05cm);
\draw[line width=.03cm,->,>=stealth] (7.2,.2)  -- (9.5,.2);

\node (init) at (9.5,.2) [label=right:$\R_i({[0,T-\delta t]})$] {};

\filldraw[]  (8,-.6) circle (0.05cm);
\draw[line width=.03cm,->,>=stealth] (8,-.6)  -- (9.5,-.6);

\node (init) at (9.5,-.6) [label=right:$\R_i({[0,T]})$] {};

\filldraw[]  (6.8,-2) circle (0.05cm);
\draw[line width=.03cm,->,>=stealth] (6.8,-2)  -- (9,-2);

\node (init) at (9,-2) [label=right:$\R_i^{c_i(\tau-\delta t)}({[0,T-\tau]})$] {};

\filldraw[]  (6.8,-2.8) circle (0.05cm);
\draw[line width=.03cm,->,>=stealth] (6.8,-2.8)  -- (9,-2.8);

\node (init) at (9,-2.8) [label=right:$\R_i^{c_i(\tau)}({[0,T-\tau]})$] {};

\draw[line width=.03cm,<->,>=stealth] (5,2)  -- (5,3.5);

\node (init) at (5,3) [label=right:$c_i(\tau)$] {};

\end{tikzpicture}
} 
\end{center}
\caption{Illustration of agent's $i$ reachable sets over the horizon $[0,T]$.}  \label{fig:reach}
\end{figure}
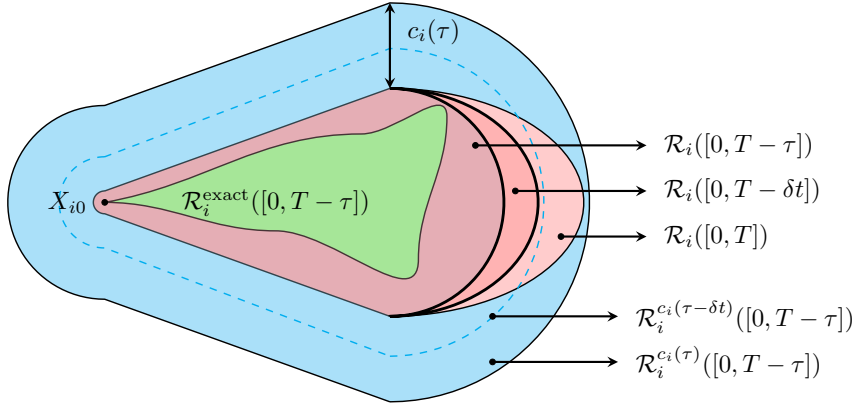

\noindent  Let $\{S_l^i\}_{l\in\bfI}$ be a cell decomposition of  $\R_i([0,T])$. Then, we define the \textit{product cell decomposition} $\{S_{\bf{l}}\}_{\bf{l}\in\bfI}$ of $\R_1([0,T])\times\cdots\times\R_N([0,T])$ as the set $\S=\{S_{\bf{l}}\}_{\bf{l}\in\bfI}:=\{S_l^1\}_{l\in\I_1}\times\cdots\times\{S_l^N\}_{l\in\I_N}$, with $\bfI:=\I_1\times\cdots\times\I_N$. Given a cell decomposition $\{S_{\bf{l}}\}_{\bf{l}\in\bfI}$ of $\R_1([0,T])\times\cdots\times\R_N([0,T])$, we use the notation $\bf{l}_i=(l_i,l_{j_1},\ldots,l_{j_{N_i}})\in\bfI_i:=\I_i\times\I_{j_1}\times\cdots\times\I_{j_{N_i}}$ to denote the indices of the cells where agent $i$ and its neighbors belong at a certain time instant (e.g., at $t=0$) and call it the (initial) cell configuration of agent $i$. Similarly, we use the notation $\bf{l}=(l_1,\ldots,l_N)\in\bfI$ to specify the indices of the cells where all the $N$ agents belong at a given time instant and call it a \textit{global} cell configuration. Thus, given a global cell configuration $\bf{l}$ it is possible to determine the cell configuration $\bf{l}_i$ of agent $i$ through the mapping ${\rm pr}_i:\bfI\to\bfI_i$, namely $\bf{l}_i={\rm pr}_i(\bf{l})$ (see Section 2 for the definition of ${\rm pr}_i(\cdot)$). Before defining the notion of a well posed space-time discretization for the overapproximations of the agents' reachable sets, we provide a class of hybrid feedback laws which are assigned to the free inputs $v_i$ in order to obtain meaningful discrete transitions. The control laws are parameterized by the agents' initial conditions and a set of auxiliary parameters which are responsible for the agent's reachability capabilities. The specific control laws of this class which are exploited for the derivation of the discretizations in this report are provided in the next section.

\begin{dfn}\label{control:class}
\noindent Consider an agent $i\in\N$, cell decompositions $\S_i=\{S_l^i\}_{l\in\I_i}$, $\S_{\kappa}=\{S_l^{\kappa}\}_{l\in\I_{\kappa}}$ of $\R_i([0,T])$ and $\R_{\kappa}([0,T])$, $\kappa\in\N_i$, respectively, a nonempty subset $W_i$ of $\Rat{n}$, and an initial cell configuration $\bf{l}_i$ of $i$. For each $x_{i0}\in S_{l_i}^i$ and $w_i\in W_i$, consider the mapping  $k_{i,\bf{l}_i}(\cdot,\cdot,\cdot;x_{i0},w_i):[0,\infty)\times \Rat{(N_i+1)n}\to\Rat{n}$, parameterized by $x_{i0}\in S_{l_i}^i$ and $w_i\in W_i$. We say that $k_{i,\bf{l}_i}(\cdot)$ satisfies \textit{Property (P)}, if the following conditions are satisfied.

\noindent\textit{(P1)} The mapping $k_{i,\bf{l}_i}(t,x_i,\bf{x}_j;x_{i0},w_i)$ is continuous on $[0,\infty)\times \Rat{(N_i+1)n}\times S_{l_i}^i\times W_i$.

\noindent\textit{(P2)} The mapping $k_{i,\bf{l}_i}(t,\cdot,\cdot;x_{i0},w_i)$ is globally Lipschitz 
continuous on $(x_i,\bf{x}_j)$ (uniformly with respect to $t\in[0,\infty)$, $x_{i0}\in S_{l_i}^i$ and $w_i\in W_i$). $\triangleleft$
\end{dfn}

\noindent We are now in position to formalize our requirement which describes the possibility for an agent to perform a discrete transition, based on the knowledge of its neighbors' discrete positions. In order to provide the definition of well posed transitions for the individual agents, we will consider for each agent $i\in\N$ the following system with disturbances:
\begin{equation} \label{system:disturbances}
\dot{x}_i=g_i(x_i,\bf{d}_j)+v_i,
\end{equation}

\noindent where $d_{j_1},\ldots,d_{j_{N_i}}:[0,\infty)\to\Rat{n}$ (also denoted $d_{\kappa}$, $\kappa\in\N_i$) are continuous functions. The use of this auxiliary system is inspired by the approach in \cite{GaMs12}, where piecewise affine systems with disturbances are exploited for the construction of symbolic models for general nonlinear systems. The map $g_i(\cdot)$ constitutes a bounded Lipchitz extension of the restriction of $f_i(\cdot)$ on $\R_i([0,T])\times\R_{j_1}([0,T])\times\cdots\times\R_{j_{N_i}}([0,T])$ satisfying 
\begin{align}
|g_i(x_i,\bf{x}_j)| & \le M(i),\forall (x_i,\bf{x}_j)\in\Rat{(N_i+1)n}  \label{function:g} \\
|g_i(x_i,\bf{x}_j)-g_i(x_i,\bf{y}_j)| & \le L_1(i)|\bf{x}_j-\bf{y}_i|, \label{dynamics:bound1} \\
|g_i(x_i,\bf{x}_j)-g_i(y_i,\bf{x}_j)| & \le L_2(i)|x_i-y_i|, \label{dynamics:bound2} \\
\forall x_i,y_i \in \R_i^{c_i(\tau)}([0,T-\tau]),\bf{x}_j,\bf{y}_j & \in\R_{j_1}([0,T])\times\cdots\times\R_{j_{N_i}}([0,T]), \nonumber
\end{align}

\noindent with $M(i)$ as given in \eqref{dynamics:bound} and with any constants $L_1(i)$ and $L_2(i)$ such that   
\begin{align}
|f_i(x_i,\bf{x}_j)-f_i(x_i,\bf{y}_j)| & \le L_1(i)|\bf{x}_j-\bf{y}_i|, \nonumber \\
|f_i(x_i,\bf{x}_j)-f_i(y_i,\bf{x}_j)| & \le L_2(i)|x_i-y_i|, \nonumber \\
\forall x_i,y_i \in \R_i^{c_i(\tau)}([0,T-\tau]),\bf{x}_j,\bf{y}_j & \in\R_{j_1}([0,T])\times\cdots\times\R_{j_{N_i}}([0,T]). \nonumber
\end{align}

\noindent This auxiliary system is used in order to provide an overapproximation of each agent's discrete transition capabilities over the horizon, by exploiting the global bounds of the auxiliary vector field $g_i(\cdot)$. Conditions under which these transitions are also implementable by the original system \eqref{single:integrator} are given later in Lemma \ref{lemma:onestep:consistency} and its corollary. Existence of a $g_i(\cdot)$ satisfying these properties is guaranteed by the Kirszbraun Lipschitz extension theorem. Notice that the Lipschitz constants above are evaluated for $x_i$ ranging in the inflated reachable set $\R_i^{c_i(\tau)}([0,T-\tau])$. The reason for this requirement comes from the fact that the individual transition system of each agent will be based on reachability properties of the auxiliary system with disturbances over the time step $[0,\delta t]$, for initial cells lying in the overapproximation $\R_i([0,T-\delta t])$ of agent's $i$ reachable set. Since these cells may in principle contain states which are outside the exact reachable state of the agent, and the disturbances do not necessarily coincide with trajectories of its neighbors over this time interval (they are an overapproximation of these trajectories), it is possible that the solution of \eqref{system:disturbances} over $[0,\delta t]$ lies outside $\R_i([0,T])$. However, by \eqref{input:bound}, \eqref{function:g} and \eqref{constant:ci} it follows that it will lie in the larger set $\R_i^{c_i(\tau)}([0,T-\tau])$.

\begin{dfn} \label{conditionCC}
Consider an agent $i\in\N$, cell decompositions $\S_i=\{S_l^i\}_{l\in\I_i}$, $\S_{\kappa}=\{S_l^{\kappa}\}_{l\in\I_{\kappa}}$ of $\R_i([0,T])$ and $\R_{\kappa}([0,T])$, $\kappa\in\N_i$, respectively, a time step $\delta t<\tau$ and assume that $\S_i$ is compliant with $\R_i([0,T-\delta t])$. Also, consider a nonempty subset $W_i$ of $\Rat{n}$, a cell configuration $\bf{l}_i$ of $i$ with $S_{l_i}^i\subset \R_i([0,T-\delta t])$, a control law
\begin{equation} \label{feedback:for:i}
v_i=k_{i,\bf{l}_i}(t,x_{i},\bf{x}_j;x_{i0},w_i)
\end{equation}

\noindent as in Definition \ref{control:class} that satisfies Property (P), and a cell decomposition $\S_i'=\{S_l^i\}_{l\in\I_i'}$  of $\R_i^{c_i(\tau)}([0,T-\tau])$ with $\S_i'\supset\S_i$, $\I_i'\supset\I_i$ and compliant with $\R_i([0,T])$. Given a vector $w_i\in W_i$, and a cell index $l_i'\in\I_i'$, we say that the \textit{Consistency Condition} is satisfied if the following hold. There exists a point $x_i'\in S_{l_i'}^i$, such that for each initial condition $x_{i0}\in S_{l_i}^i$ and selection of continuous functions $d_{\kappa}:\RgeO\to\Rat{n}$, $\kappa\in\N_i$ satisfying
\begin{equation} \label{disturbance:bounds}
d_{\kappa}(t)\in (S_{l_{\kappa}}^{\kappa}+B((M(\kappa)+v_{\max}(\kappa))t))\cap\R_{\kappa}([0,T]),\forall\kappa\in\N_i,t\in [0,\delta t],
\end{equation}

\noindent the solution $x_i(\cdot)$ of the system with disturbances \eqref{system:disturbances} with $v_i=k_{i,\bf{l}_i}(t,x_{i},\bf{d}_j;x_{i0},w_i)$, satisfies

\begin{equation} \label{xi:consistency:bounds}
d(x_i(t),S_{l_i}^i)<(M(i)+v_{\max}(i))t, \forall t\in (0,\delta t].
\end{equation}

\noindent Furthermore, it holds
\begin{align} 
x_i(\delta t)= & x_i'\in S_{l_i'}^i, \label{xi:in:finalcell} \\
|k_{i,\bf{l}_i}(t,x_{i}(t),\bf{d}_j(t);x_{i0},w_i)|< & v_{\max}(i),\forall  t\in[0,\delta t].  \label{controler:consistency} \quad\triangleleft
\end{align}
\end{dfn}

Notice that when the Consistency Condition is satisfied, agent $i$ can be driven to cell $S_{l_i'}^i$ precisely in time $\delta t$ under the auxiliary dynamics \eqref{system:disturbances}, with the feedback law $k_{i,\bf{l}_i}(\cdot)$ corresponding to the given parameter $w_i$ in the definition. The latter is possible for all disturbances which satisfy \eqref{disturbance:bounds} and capture the possibilities for the evolution of $i$'s neighbors over the time interval $[0,\delta t]$, given the knowledge of its neighbors' cell configuration. Under some additional asumptions, which are provided in Lemma \ref{lemma:onestep:consistency}, the latter transitions can be also implemented by the original system \eqref{single:integrator} and the control law $k_{i,\bf{l}_i}(\cdot)$. We proceed with the definition of a well posed online abstraction for each agent in order to extract a finite transition system. Note that due to the boundedness of the reachable sets over the time horizon we can always select finite overapproximations of the agents' reachable sets which result in compliant cell decompositions in Definition \ref{conditionCC} and thus, derive a finite transition system (by selecting e.g., $\R_i(t)={\rm int}(B(\beta(T+1,|\bf{X}_0|)))$ for all $t\in[0,T]$).

\begin{dfn}\label{well:posed:discretization}
Consider cell decompositions $\S_i=\{S_l^i\}_{l\in\I_i}$ of $\R_i([0,T])$, $i\in\N$, their product decomposition $\S$, a time step $\delta t<\tau$ with $T=\ell\delta t$, nonempty subsets $W_i$, $i\in\N$ of $\Rat{n}$ and assume that each $\S_i$ is compliant with  $\R_i([0,T-\delta t])$.

\noindent \textit{(i)} Given an agent $i\in\N$, a cell decomposition $\S_i'=\{S_l^i\}_{l\in\I_i'}$ of $\R_i^{c_i(\tau)}([0,T-\tau])$ with $\S_i'\supset\S_i$,  $\I_i'\supset\I_i$ and compliant with $\R_i([0,T])$, an initial cell configuration $\bf{l}_i\in\bfI_i$ of $i$ with $S_{l_i}^i\subset\R_i([0,T-\delta t])$, and a cell index $l_i'\in\I_i'$, we say that \textit{the transition} $l_i\overset{\bf{l}_i}{\longrightarrow}l_i'$ \textit{is well posed with respect to the space-time discretization $\S-\delta t$}, if there exist a feedback law  $v_i=k_{i,\bf{l}_i}(\cdot,\cdot,\cdot;x_{i0},w_i)$ as in Definition \ref{control:class} that satisfies Property (P), and a vector $w_i\in W_i$, such that the Consistency Condition of Definition \ref{conditionCC} is fulfilled.

\noindent \textit{(ii)} We say that the \textit{space-time discretization} $\S-\delta t$ \textit{is well posed}, if for each agent $i\in\N$, cell decomposition $\S_i'=\{S_l^i\}_{l\in\I_i'}$ of $\R_i^{c_i(\tau)}([0,T-\tau])$ with $\S_i'\supset\S_i$,  $\I_i'\supset\I_i$ and compliant with $\R_i([0,T])$, and cell configuration $\bf{l}_i$ of $i$, there exists a cell index $l_i'\in\I_i'$ such that the transition $l_i\overset{\bf{l}_i}{\longrightarrow}l_i'$ is well posed with respect to $\S-\delta t$.
\end{dfn}

Based on Definitions~\ref{well:posed:discretization}(i), we proceed by defining the discrete transition system which serves as an abstract model for the behavior of each agent. The transitions are established through the verification of the Consistency Condition which exploits the auxiliary system with disturbances \eqref{system:disturbances}.  

\begin{dfn} \label{individual:ts}
Consider cell decompositions $\S_i=\{S_l^i\}_{l\in\I_i}$ of $\R_i([0,T])$, $i\in\N$, their product decomposition $\S$, a time step $\delta t<\tau$ with $T=\ell\delta t$, nonempty subsets $W_i$, $i\in\N$ of $\Rat{n}$ and assume that each $\S_i$ is compliant with  $\R_i([0,T-\delta t])$.  The \textit{individual transition system} $TS_i:=(Q_i,Q_{0i},Act_i,$ $\longrightarrow_i)$ of each agent $i\in\N$ is defined as:

\noindent \textbullet \; State set $Q_i:=\I_i$ (the indices of the cell decomposition $\S_i$);

\noindent \textbullet \; Initial state set $Q_{0i}:=\{l_i\in\I_i:X_{i0}\in S_{l_i}^i\}$;

\noindent \textbullet \; Actions $Act_i:=\bfI_i$ (the cell configurations of $i$);

\noindent \textbullet \; Transition relation $\longrightarrow_i\subset Q_i\times Act_i\times Q_i$ defined as follows. For any $l_i,l_i'\in Q$ and $\bf{l}_i=(l_i,l_{j_1},\ldots,l_{j_{N_i}})\in\bfI_i$, $l_i\overset{\bf{l}_i}{\longrightarrow_i}l_i'$, iff $l_i\overset{\bf{l}_i}{\longrightarrow}l_i'$ is well posed (implying also that $S_{l_i}^i\subset\R_i([0,T-\delta t])$).
\end{dfn}

\begin{rem}
Notice that the auxiliary cell decomposition $\S_i'$ with indices $\I_i'$ which is exploited for the verification of the Consistency Condition can provide according to Definition~\ref{well:posed:discretization}(i) well posed transitions which lead to a cell $S_{l_i'}^i$ outside $\R_i([0,T])$. These transitions are excluded from the definition of each agent's individual transition system, since they do not capture any possible behavior of the system over $[0,T]$, and the state set of the transition system contains only the indices of $\I_i$, namely, of the original cell decomposition $\S_i$. In particular, the transitions of possible interest over the horizon are the ones where the initial and final state of the agent lie in the exact reachable sets over $[0,T-\delta t]$ and $[0,T]$, respectively. Since the latter cannot be computed explicitly in principle, we impose this requirement for their overapproximations $\R_i([0,T-\delta t])$ and $\R_i([0,T])$. 
In addition, for the case where the cells of an agent and its neighbors have nonempty intersection with the corresponding agents' reachable cells at certain time instant $t=m\delta t$ with $m\in\{0,\ldots,\ell-1\}$, it will be shown in Corollary \ref{corollary:individual:post} that there is always an outgoing transition for well posed discretizations. 
\end{rem}

In the subsequent analysis we will consider well posed discretizations which implies that their time step $\delta t$ has been selected so that $T=\ell\delta t$ and will focus on transition sequences of length $m\le\ell$ originating from cells which contain the agents' initial positions $X_{i0}$, $i\in\N$. 
Such sequences are defined below for the individual transition system of each agent. In addition, it will be shown in the sequel that the projection of a transition sequence originating from the discrete state containing $\bf{X}_0$ in the product discrete model (of all agents) to the individual transition system of each agent will provide such a sequence of transitions for each agent, which can also be implemented by the continuous time system. 

\begin{dfn} \label{dfn:strongly:well:posed:seq}
Consider cell decompositions $\S_i=\{S_l^i\}_{l\in\I_i}$ of $\R_i([0,T])$, $i\in\N$, their product decomposition $\S$, a time step $\delta t<\tau$ with $T=\ell\delta t$, nonempty subsets $W_i$, $i\in\N$ of $\Rat{n}$ and assume that each $\S_i$ is compliant with  $\R_i([0,T-\delta t])$. Given an agent $i\in\N$, an integer $m\in\{1,\ldots,\ell\}$, cell configurations $\bf{l}^{\kappa}=(l_i^{\kappa},l_{j_1}^{\kappa},\ldots,l_{j_{N_i}}^{\kappa})\in\bfI_i$, $\kappa=0,\ldots,m-1$ of $i$ and a cell index $l_i^m\in\I_i$, we say that  $\bf{l}_i^0\bf{l}_i^1\cdots\bf{l}_i^{m-1}l_i^m$ is \textit{a strongly well posed transition sequence of order $m$ (with respect to $\S-\delta t$)}, if it holds $X_{i0}\in S_{l_i^0}^i$, and  $l_i^{\kappa}\overset{\bf{l}_i^{\kappa}}{\longrightarrow_i}l_i^{\kappa+1}$. We also define $\bf{l}_i^0$ as \textit{a strongly well posed transition sequence of order 0} if $X_{i0}\in S_{l_i^0}^i$.
\end{dfn}

\noindent The following lemma establishes that for well posed discretizations and cell configurations of all agents which intersect their exact reachable cells at a certain time instant $t=m\delta t$ with $m\in\{0,\ldots,\ell-1\}$ there exists a transition for each agent that can be implemented by the continuous time system \eqref{single:integrator}.  

\begin{lemma} \label{lemma:onestep:consistency}
Consider cell decompositions $\S_i=\{S_l^i\}_{l\in\I_i}$ of $\R_i([0,T])$, $i\in\N$, their product $\S$, a time step $\delta t<\tau$ with $T=\ell\delta t$, nonempty subsets $W_i$, $i\in\N$ of $\Rat{n}$ and assume that each $\S_i$ is compliant with  $\R_i([0,T-\delta t])$ and that the space-time discretization $\S-\delta t$ is well posed. Also, consider a cell configuration $\bf{l}=(l_1,\ldots,l_N)$, an integer $m\in\{0,\ldots,\ell-1\}$, an input $v=(v_1,\ldots,v_N)\in\U$ and assume that each component $x_i(\cdot,\bf{X}_0;v)$ of the solution of \eqref{single:integrator} satisfies $x_i(m\delta t,\bf{X}_0;v)\in S_{l_i}^i$. 

\noindent \textit{(i)} Then, it holds that ${\rm Post}_i(l_i,{\rm pr}_i(\bf{l}))\ne\emptyset$ for all $i\in\N$. In particular, 
\begin{equation} \label{post:indiv:characterization}
{\rm Post}_i(l_i,{\rm pr}_i(\bf{l}))=\{l_i'\in\I_i':l_i\overset{\bf{l}_i}{\longrightarrow}l_i'\;\textup{is well posed}\}\subset\I_i,
\end{equation}

\noindent for any cell decomposition  $\S_i'=\{S_l^i\}_{l\in\I_i'}$ of $\R_i^{c_i(\tau)}([0,T-\tau])$ with $\S_i'\supset\S_i$,  $\I_i'\supset\I_i$ and compliant with $\R_i([0,T])$, and is uniquely defined, irrespectively of the cell decomposition $\S_i'$. 

\noindent \textit{(ii)} In addition, for any selection of $l_i'\in{\rm Post}_i(l_i,{\rm pr}_i(\bf{l}))$, $i\in\N$, the following hold. There exist feedback laws $v_i=k_{i,\bf{l}_i}$ as in \eqref{feedback:for:i} and vectors $w_i\in W_i$ for all $i\in\N$ such that the solution $\xi(\cdot)$ of the closed loop system \eqref{single:integrator}, \eqref{feedback:for:i} with initial condition $\xi(0)=x(m\delta t,\bf{X}_0;v)$, $i\in\N$ satisfies $\xi_i(\delta t)\in S_{l_i'}^i$ and $|k_{i,\bf{l}_i}(t,\xi_{i}(t),\mbf{\xi}_j(t);x_{i0},w_i)|\le  v_{\max}(i)$ for all $t\in[0,\delta t]$ and $i\in\N$. Furthermore, there exists $u\in\U$ with $u(t)=v(t)$ for all $t\in[0,m\delta t)$, such that the solution of \eqref{single:integrator} satisfies $x_i((m+1)\delta t,\bf{X}_0;u)\in S_{l_i'}^i$ for all $i\in\N$.
\end{lemma}

\begin{proof}[Proof of (i)]
Consider for each agent $i\in\N$ a cell decomposition $\S_i'=\{S_{l}^i\}_{l\in\I_i'}$ of $\R_i^{c_i(\tau)}[0,T-\tau])$ with $\S_i'\supset \S_i$, $\I_i'\supset \I_i$, and compliant with $\R_i([0,T])$. Without any loss of generality we will assume that $m>0$, since the analysis when $m=0$ constitutes a special case of the proof below. Also, since $m\le\ell-1$ we have for each $i\in\N$ that $x_i(m\delta t,\bf{X}_0;v)\in\R_i([0,T-\delta t])$, implying that $S_{l_i}^i\cap \R_i([0,T-\delta t])\ne \emptyset$ and thus, since $\S_i$ is compliant with $\R_i([0,T-\delta t])$, that $S_{l_i}^i\subset\R_i([0,T-\delta t])$. Hence, since the space-time discretization is well posed, given the cell configuration ${\rm pr}_i(\bf{l})$ of $i$, there exists a cell index $l_i'\in\I_i'$ such that the transition $l_i\overset{\bf{l}_i}{\longrightarrow}l_i'$ is well posed. The latter implies existence of a feedback law $v_i=k_{i,\bf{l}_i}(t,x_{i},\bf{d}_j;x_{i0},w_i)$ (see \eqref{feedback:for:i}) and $x_i'\in S_{l_i'}^i$ such that the requirements of the Consistency Condition are fulfilled for the system with disturbances \eqref{system:disturbances}. Next, we select for each agent the initial condition $x_{i0}=x_i(m\delta t,\bf{X}_0;v)$ and denote by $\xi(\cdot)$ the solution of the closed loop system \eqref{single:integrator} with $v_i=k_{i,\bf{l}_i}$, $i\in\N$ as selected above. By the local Lipschitz property on the $f_i$'s and $k_{i,\bf{l}_i}$'s, it follows that the system has a unique solution defined on the right maximal interval $[0,T_{\max})$. We will show that $T_{\max}>\delta t$, and that for each agent $i\in\N$, the $i$-th component of the solution coincides with the solution of system \eqref{system:disturbances} on $[0,\delta t]$ with $\bf{d}_j=\mbf{\xi}_j$ and the same initial condition $x_{i0}$.

Notice that since $\xi_i(0)=x_{i0}\in S_{l_i}^i\subset \R_i([0,T-\delta t])$ and $\xi_{\kappa}(0)=x_{\kappa 0}\in S_{l_{\kappa}}^{\kappa}\subset \R_{\kappa}([0,T-\delta t])$, $\kappa\in\N_i$, it follows that the requirements of the Consistency Condition are fulfilled with $\bf{d}_j=\mbf{\xi}_j$ and $x_i=\xi_i$ at $t=0$, which implies that the control laws $k_{i,\bf{l}_i}$ satisfy $|k_{i,\bf{l}_i}(0,\xi_i(0),\mbf{\xi}_j(0);x_{i0},w_i)|<v_{\max}(i)$ for all $i\in\N$. The latter in conjunction with continuity of the functions $\xi_i(\cdot)$,  $k_{i,\bf{l}_i}(\cdot)$, the fact that $\xi_i(0)=x_{i0}\in S_{l_i}^i\subset \R_i([0,T])$ which is open, \eqref{system:disturbances}, and \eqref{function:g}, implies that there exists $\delta>0$ such that 
\begin{align}
\xi_i(t) & \in\R_i([0,T]), \forall t\in[0,\delta],i\in\N, \label{xi:property1} \\
d(\xi_i(t),S_{l_i}^i) & < (M(i)+v_{\max}(i))t, \forall t\in[0,\delta],i\in\N. \label{xi:property2}
\end{align}

\noindent We claim that the latter properties hold for all  $t\in[0,\delta t]\cap[0,T_{\max})$. Indeed, assume on the contrary that there exist a time $\sigma\in(0,\delta t]\cap(0,T_{\max})$ and an agent $i\in\N$, such that  
\begin{equation} \label{sigma:contradiction}
\xi_i(\sigma)\notin\R_i([0,T])\;{\rm or}\; d(\xi_i(\sigma),S_{l_i}^i) \ge (M(i)+v_{\max}(i))\sigma.
\end{equation}

\noindent Then, if we define 
\begin{align} 
\bar{t}:=\sup\{t\in(0,\delta t]\cap(0,T_{\max}):\xi_i(s) & \in\R_i([0,T])\;{\rm and}\; \nonumber \\
d(\xi_i(s),S_{l_i}^i)< (M(i) & +v_{\max}(i))s,\forall s\in(0,t],i\in\N\}, \label{time:tbar}
\end{align}

\noindent it follows from \eqref{xi:property1}, \eqref{xi:property2} and \eqref{sigma:contradiction} that 
\begin{equation} \label{tbar:properties}
0<\bar{t}\le\delta t;\quad \bar{t}<T_{\max},
\end{equation}

\noindent where the right hand sides of these inequalities are based on the fact that $\sigma\le\delta t$ and $\sigma<T_{\max}$. In addition, by recalling that $\R_i([0,T])$ is open and invoking continuity of $\xi(\cdot)$, it follows from \eqref{time:tbar} that there exists an agent $\iota\in\N$  such that 
\begin{equation} \label{iota:contradiction}
\xi_{\iota}(\bar{t})\notin\R_{\iota}([0,T])\;{\rm or}\; d(\xi_{\iota}(\bar{t}),S_{l_{\iota}}^{\iota}) = (M(\iota)+v_{\max}(\iota))\bar{t}.
\end{equation}

\noindent We proceed by considering for each agent $i\in\N$ the solution $\zeta_i(\cdot)$ of the system with disturbances \eqref{system:disturbances} with $\bf{d}_j=\mbf{\xi}_j$, initial condition $x_{i0}$ and the selected $w_i$ at the beginning of the proof. Then, for each $i\in\N$ it follows from \eqref{time:tbar} that each disturbance $d_{\kappa}(\cdot)$, $\kappa\in\N_i$ satisfies \eqref{disturbance:bounds} for all $t\in[0,\bar{t})$. Thus, it follows from causality of the solution $\zeta_i(\cdot)$ with respect to the disturbances and the Consistency Condition that
\begin{align}
\zeta_i(t) & \in\R_i([0,T]), \forall t\in[0,\bar{t}],i\in\N, \label{zetai:property1} \\
d(\zeta_i(t),S_{l_i}^i) & < (M(i)+v_{\max}(i))t, \forall t\in[0,\bar{t}],i\in\N, \label{zetai:property2} \\
|k_{i,\bf{l}_i}(t,\zeta_{i}(t),\mbf{\xi}_j(t);x_{i0},w_i)| & \le v_{\max}(i), \forall t\in[0,\bar t], i\in\N. \label{zetai:property3}
\end{align}

\noindent In addition, since $g_i(x_i,\bf{x}_j)=f_i(x_i,\bf{x}_j)$ for all $x_i\in\R_i([0,T])$, $x_{\kappa}\in\R_{\kappa}([0,T])$, $\kappa\in\N_i$, and by virtue of \eqref{time:tbar} and \eqref{tbar:properties} the solution $\xi(\cdot)$ is defined on $[0,\bar{t}]$ and satisfies $\xi_i(t)\in\R_i([0,T])$ for all $t\in[0,\bar{t})$ and $i\in\N$, it follows from \eqref{zetai:property1} and standard ODE arguments that $\zeta_i(\cdot)$ and $\xi_i(\cdot)$ coincide on $[0,\bar{t}]$ for all $i\in\N$. Hence, we obtain from \eqref{zetai:property2} that $ d(\xi_{\iota}(\bar{t}),S_{l_{\iota}}^{\iota}) < (M(\iota)+v_{\max}(\iota))\bar{t}$. Thus, in order to reach a contradiction with \eqref{iota:contradiction}, we need to also prove that $\xi_{\iota}(\bar{t})\in\R_{\iota}([0,T])$. Notice first, that since $\zeta_i(t)=\xi_i(t)$ for all $t\in[0,\bar{t}]$, it follows from \eqref{zetai:property3} that 
\begin{equation} \label{control:laws:along:solutions}
|k_{i,\bf{l}_i}(t,\xi_{i}(t),\mbf{\xi}_j(t);x_{i0},w_i)|\le v_{\max}(i), \forall t\in[0,\bar t], i\in\N.
\end{equation}
 
\noindent Thus, it turns out that if we select the input $\omega=(\omega_1,\ldots,\omega_N):\RgeO\to\Rat{Nn}$ such that for each $i\in\N$ it holds 
\begin{align*}
\omega_i(t):= & k_{i,\bf{l}_i}(t,\xi_{i}(t),\mbf{\xi}_j(t);x_{i0},w_i), t\in[0,\bar{t}), \\
|\omega_i(t)|\le & v_{\max}(i), t\ge\bar{t},
\end{align*} 

\noindent we obtain that $|\omega_i(t)|\le v_{\max}(i)$, $\forall t\ge 0$ and hence, that $\omega\in\U$. In addition, since by \eqref{tbar:properties} the solution $\xi(\cdot)$ is defined on $[0,\bar{t}]$, it follows from standard ODE arguments that 
\begin{equation} \label{xi:eq:eta}
\xi(t)=\eta(t,x_0;\omega),\forall t\in[0,\bar{t}),
\end{equation}

\noindent with $\eta(\cdot)$ being the solution of \eqref{single:integrator} with input $\omega(\cdot)$ and initial condition $x_0$. Define now $u=(u_1,\ldots,u_N):\RgeO\to\Rat{Nn}$ with 
\begin{align}
u_i(t):= & v_i(t), t\in[0,m\delta t), \nonumber \\
u_i(t):= & \omega_i(t), t\in[m\delta t,m\delta t+\bar{t}), \nonumber \\
|u_i(t)|\le & v_{\max}(i), t\ge m\delta t+\bar{t}, \label{input:ui}
\end{align} 

\noindent for all $i\in\N$ and notice that $u\in\U$. Next, by using the transition map $\varphi(\cdot)$ introduced for system \eqref{single:integrator} at the beginning of the section, we have that $\eta(t,x_0;\omega)=\varphi(t,0,x_0;\omega)$, $\forall t\in[0,\bar{t}]$ and by time invariance, causality and \eqref{input:ui}, which implies that $\omega(t)={\rm Sh}_{m\delta t}(u(t))$, $\forall t\in[0,\bar{t})$, that $\varphi(t,0,x_0;\omega)=\varphi(m\delta t+t,m\delta t,x_0;u)$, $\forall t\in[0,\bar{t}]$. In addition, since by causality and \eqref{input:ui} it holds  $x_0=\varphi(m\delta t,0,\bf{X}_0;u)$, we deduce from the semigroup property that $\varphi(t,0,x_0;\omega)=\varphi(m\delta t+t,m\delta t,\varphi(m\delta t,0,\bf{X}_0;u);u)=\varphi(m\delta t+t,0,\bf{X}_0;u)=x(m\delta t+t,\bf{X}_0;u)$, $\forall t\in[0,\bar{t}]$, and thus, it follows from the above derivations that $\eta(t,x_0;\omega)=x(m\delta t+t,\bf{X}_0;u)$, $\forall t\in[0,\bar{t}]$. Hence, we have from the latter and \eqref{xi:eq:eta} that  
\begin{equation} \label{xi:eq:ex}
\xi(t)=x(m\delta t+t,\bf{X}_0;u), \forall t\in[0,\bar{t}].
\end{equation}

\noindent Since $\bar{t}\le\delta t$ and $m<\ell$, implying that $m\delta t+\bar{t}\le\ell\delta t=T$, and hence, that $x_{\iota}(m\delta t+\bar{t},\bf{X}_0;u)\in\R_{\iota}([0,T])$, we obtain from \eqref{xi:eq:ex} that also $\xi_{\iota}(\bar{t})\in\R_{\iota}([0,T])$, which establishes the desired contradiction. Consequently, we deduce that \eqref{xi:property1} and \eqref{xi:property2} hold for all  $t\in[0,\delta t]\cap[0,T_{\max})$. We next show that $T_{\max}>\delta t$. Indeed, assume on the contrary that $T_{\max}\le\delta t$. By reasoning as above, we obtain that \eqref{xi:eq:ex} holds for all $t\in[0,T_{\max})$ and $i\in\N$, i.e., that $\xi(t)=x(m\delta t+t,\bf{X}_0;u)$, for all $t\in[0,T_{\max})$. Since $x(\cdot,\bf{X}_0;u)$ is defined for all positive times by forward completeness of \eqref{system:disturbances}, we obtain that $\lim_{t\nearrow T_{\max}}\xi(t)=x(m\delta t+T_{\max},\bf{X}_0;u)$, contradicting maximality of $[0,T_{\max})$. Hence, having established that $T_{\max}>\delta t$, we get that \eqref{xi:property1} and \eqref{xi:property2} hold for all $t\in[0,\delta t]$, implying that 
\begin{equation} \label{xiatdeltat:inRi} 
\xi_i(\delta t)\in\R_i([0,T]),\forall i\in\N.
\end{equation}

\noindent In addition, for each $i\in\N$ we deuce from the fact that \eqref{xi:property1} and \eqref{xi:property2} hold for all $t\in[0,\delta t]$ and the Consistency Condition applied to the system with disturbances \eqref{system:disturbances} with $\bf{d}_j=\mbf{\xi}_j$, initial condition $x_{i0}$, the selected $w_i$ at the beginning of the proof, and similar arguments as above, that 
\begin{equation}\label{xiatdeltat:inSli}
\xi_i(\delta t)=x_i'\in S_{l_i'}^i,\forall i\in\N.
\end{equation}

\noindent The latter, by virtue of \eqref{xiatdeltat:inRi} and the fact that $\S_i'$ is compliant with $\R_i([0,T])$, $\S_i'\supset S_i$ and $\I_i'\supset \I_i$ implies that $l_i'\in \I_i$. Indeed, since $ S_{l_i'}^i\cap\R_i([0,T])\ne\emptyset$ we have from compliance that $S_{l_i'}^i\subset\R_i([0,T])$. Hence, by the definition of a cell decomposition, there exists $S_{l_i}^i\in\S_i$ with ${\rm int}(S_{l_i}^i)\cap{\rm int}(S_{l_i'}^i)\ne\emptyset$. Since $\S_i'\supset \S_i$, implying that also $S_{l_i}^i\in\S_i'$ it follows again from the definition of a cell decomposition that necessarily $S_{l_i}^i=S_{l_i'}^i$ and thus, from the fact that $\I_i'\supset \I_i$, we get that $l_i'\in \I_i$. From the latter and by recalling that $S_{l_i}^i\subset\R_i([0,T-\delta t])$ and  $l_i\overset{\bf{l}_i}{\longrightarrow}l_i'$ is well posed, we deduce that $l_i\overset{\bf{l}_i}{\longrightarrow_i}l_i'$, implying that ${\rm Post}_i(l_i,{\rm pr}_i(\bf{l}))\ne\emptyset$. In addition, since $l_i'$ was selected as an arbitrary cell for which $l_i\overset{\bf{l}_i}{\longrightarrow}l_i'$ is well posed, we deuce that \eqref{post:indiv:characterization} also holds.

\noindent \textit{Proof of (ii).} The result of Part (ii) follows from the proof of Part (i), by selecting the feedback laws $k_{i,\bf{l}_i}$ as in Part (i) which establishes \eqref{control:laws:along:solutions} for all $t\in[0,\delta t]$, and $u(\cdot)$ as in \eqref{input:ui}, which by virtue of \eqref{xiatdeltat:inSli} and the fact that \eqref{xi:eq:ex} can be verified for all $t\in[0,\delta t]$, implies that $x_i((m+1)\delta t,\bf{X}_0;u)\in S_{l_i'}^i$ for all $i\in\N$.
\end{proof}

\begin{corollary} \label{corollary:individual:post}
Consider cell decompositions $\S_i=\{S_l^i\}_{l\in\I_i}$ of $\R_i([0,T])$, $i\in\N$, their product $\S$, a time step $\delta t<\tau$ with $T=\ell\delta t$, nonempty subsets $W_i$, $i\in\N$ of $\Rat{n}$ and assume that each $\S_i$ is compliant with  $\R_i([0,T-\delta t])$ and that the space-time discretization $\S-\delta t$ is well posed. Also, consider an agent $i$, a cell configuration $\bf{l}_i$ of $i$, an integer $m\in\{0,\ldots,\ell-1\}$, an input $v\in\U$ and assume that each component $x_{\kappa}(\cdot,\bf{X}_0;v)$, $\kappa\in\N_i\cup\{i\}$ of the solution of \eqref{single:integrator} satisfies $x_{\kappa}(m\delta t,\bf{X}_0;v)\in S_{l_{\kappa}}^{\kappa}$. 
Then, it holds that ${\rm Post}_i(l_i,\bf{l}_i)=\{l_i'\in\I_i':l_i\overset{\bf{l}_i}{\longrightarrow}l_i'\;\textup{is well posed}\}\subset\I_i$ for any cell decomposition $\S_i'$ of $\R_i^{c_i(\tau)}([0,T-\tau])$ as in Lemma \ref{lemma:onestep:consistency}.
\end{corollary}

\begin{proof}
Indeed, since $x_{\kappa}(m\delta t,\bf{X}_0;v)\in \R_{\kappa}([0,T-\delta t])$ for all $\kappa\in\N$, we may select cells $S_{l_{\kappa}}^{\kappa}\in\S_{\kappa}$, with  $x_{\kappa}(m\delta t,\bf{X}_0;v)\in S_{l_{\kappa}}^{\kappa}$ for all $\kappa\in\N\setminus(\N_i\cup\{i\})$. Then, the result follows from Lemma \ref{lemma:onestep:consistency} by considering the cell configuration $\bf{l}=(l_1,\ldots,l_N)$ with ${\rm pr}_i(\bf{l})=\bf{l}_i$ and $l_{\kappa}$,  $\kappa\in\N\setminus(\N_i\cup\{i\})$ as previously selected.
\end{proof}

\noindent Based on the result of Lemma \ref{lemma:onestep:consistency} we can show that consistent discrete sequences of all agents which project to a strongly well posed transition sequence for each agent, have always outgoing transitions. 

\begin{prop} \label{proposition:implementation:of:strwellposed}
Consider cell decompositions $\S_i=\{S_l^i\}_{l\in\I_i}$ of $\R_i([0,T])$, $i\in\N$, their product $\S$, a time step $\delta t<\tau$ with $T=\ell\delta t$, nonempty subsets $W_i$, $i\in\N$ of $\Rat{n}$ and assume that each $\S_i$ is compliant with  $\R_i([0,T-\delta t])$ and that the space-time discretization $\S-\delta t$ is well posed. Also, consider a sequence $\bf{l}^0\cdots\bf{l}^m$ of global cell configurations with $m\in\{0,\ldots,\ell\}$ such that ${\rm pr}_i(\bf{l}^0)\cdots{\rm pr}_i(\bf{l}^{m-1})l^m_i$ is a strongly well posed transition sequence of order $m$ for each $i\in\N$.

\noindent \textit{(i)} Then, there exists $v\in\U$ such that each component $x_i(\cdot,\bf{X}_0;v)$ of the solution of \eqref{single:integrator} satisfies $x_i(\kappa\delta t,\bf{X}_0;v)\in S^i_{l_i^{\kappa}}$, for all $\kappa\in\{0,\ldots,m\}$. 

\noindent \textit{(ii)} If in addition $m<\ell$, then it holds ${\rm Post}_i(l_i^m,{\rm pr}_i(\bf{l}^m))\ne\emptyset$ for all $i\in\N$.
\end{prop}   

\begin{proof}[Proof of (i)]
Notice first that by the definition of a strongly well posed transition sequence it holds that $X_{i0}\in S_{l_i^0}^i$ for all $i\in\N$. The result is shown by induction for $\kappa=0,\ldots,m$ and specifically, by proving the following Induction Hypothesis:

\noindent \textit{Induction Hypothesis.} For each $\kappa=0,\ldots,m$ there exists an input $v\in\U$ such that each component $x_i(\cdot,\bf{X}_0;v)$ of the solution of \eqref{single:integrator} satisfies $x_i(\kappa'\delta t,\bf{X}_0;v)\in S^i_{l_i^{\kappa'}}$, for all $\kappa'\in\{0,\ldots,\kappa\}$.

\noindent Then, the result is a direct consequence of the Induction Hypothesis with $\kappa=m$. We next proceed with the proof of the Induction Hypothesis. For $\kappa=0$, the result is a direct consequence of the fact that $X_{i0}=x_i(0,\bf{X}_0;v)\in S_{l_i^0}^i$ for any $v\in\U$. For the general case, assume that the Induction Hypothesis is valid for certain $\kappa\in\{0,\ldots,m-1\}$, implying that $x_i(\kappa\delta t,\bf{X}_0;v)\in S^i_{l_i^{\kappa}}$ for all $i\in\N$. From the latter property, and since  ${\rm pr}_i(\bf{l}^0)\cdots{\rm pr}_i(\bf{l}^{m-1})l^m_i$ is strongly well posed, implying that $l_i^{\kappa+1}\in {\rm Post}_i(l_i^{\kappa},{\rm pr}_i(\bf{l}^{\kappa}))$, it follows from Lemma~\ref{lemma:onestep:consistency}(ii) applied with $m=\kappa$ that there exists $u\in\U$ with $u(t)=v(t)$ for all $t\in[0,\kappa\delta t)$ such that $x_i((\kappa+1)\delta t,\bf{X}_0;u)\in S_{l_i'}^i$ for all $i\in\N$. The latter implies that the Induction Hypothesis is valid for $\kappa+1$ with $v=u$. 

\noindent \textit{Proof of (ii).} The proof of Part (ii) is a direct consequence of the result of Part (i) with $\kappa=m$ and Lemma \ref{lemma:onestep:consistency}(i).
\end{proof}

Based on Proposition \ref{proposition:implementation:of:strwellposed} we can derive the desired properties of the product transition system corresponding to the space-time discretization, which will be defined recursively. In particular, given the product $\bfI=\I_1\times\cdots\times\I_N$ of the cell indices corresponding to the decompositions of the sets $\R_i([0,T])$, $i\in\N$, we define the operator $\P:\bfI\to 2^{\bfI}$ as  
\begin{equation} \label{post:operator:def}
\P(\bf{l}):={\rm Post}_1(l_1;{\rm pr}_1(\bf{l}))\times\cdots\times {\rm Post}_N(l_N;{\rm pr}_N(\bf{l})), \bf{l}\in\bfI,
\end{equation}

\noindent where ${\rm Post}_i(\cdot;\cdot)$, $i\in\N$ are the post operators for the agent's individual transition systems. We also recursively define the operators $\P^{\kappa}:2^{\bfI}\to  2^{\bfI}$, $\kappa\in\mathbb{N}\cup\{0\}$, given for each $\I\subset\bfI$ as
\begin{align*}
\P^0(\I) & :=\I; \\
\P^{\kappa}(\I) & :=\P(\P^{\kappa-1}(\I)),\kappa\ge 1. 
\end{align*}

\noindent From this definition it follows directly that for any $\kappa\ge 1$ and $\I\subset\bfI$ it holds
\begin{equation} \label{post:operator:property}
\bf{l}'\in \P^{\kappa}(\I) \iff \exists \bf{l}\in\P^{\kappa-1}(\I)\;\textup{such that}\; \bf{l}'\in \P(\I).
\end{equation} 

\noindent We next provide the definition of the product transition system.

\begin{dfn}\label{product:TS}
\textit{(i)} Consider cell decompositions $\S_i=\{S_l^i\}_{l\in\I_i}$ of $\R_i([0,T])$, $i\in\N$, their product decomposition $\S$, a time step $\delta t<\tau$ with $T=\ell\delta t$, nonempty subsets $W_i$, $i\in\N$ of $\Rat{n}$ and assume that each $\S_i$ is compliant with  $\R_i([0,T-\delta t])$. Also, consider for each agent $i\in\N$ its individual transition system $TS_i$ as provided by Definition \ref{individual:ts}. The product transition system $TS^{\P}:=TS_1\otimes\cdots\otimes TS_N$ is the transition system $(Q,Q_0,{\rm Act},\longrightarrow)$ with

\noindent \textbullet \; State set $Q:=\bfI=\I_1\times\cdots\times \I_N$  (the indices of the product cell decomposition);

\noindent \textbullet \; Initial state set $Q_0:=Q_{10}\times\cdots\times Q_{N0}$, $Q_{0i}:=\{l_i\in\I_i:X_{i0}\in S_{l_i}^i\}$, $i\in\N$;

\noindent \textbullet \; Actions $Act:=\{*\}$;

\noindent \textbullet \; Transition relation $\longrightarrow\subset Q\times Act\times Q$ defined as follows. For any $\bf{l},\bf{l}'\in Q$, $\bf{l}\overset{*}{\longrightarrow}\bf{l}'$, iff there exists $m\in\{0,\ldots,\ell-1\}$ such that $\bf{l}\in \P^{m}(Q_0)$ and  $\bf{l}'\in \P(\bf{l})$.

\noindent \textit{(ii)} A path of length $m\in\{0,\ldots\ell\}$ originating from $\bf{l}^0 $ in $TS^{\P}$, is a finite sequence of states $\bf{l}^0 \bf{l}^1\cdots\bf{l}^{m}$ such that  $\bf{l}^0\in Q_0$ and $ \bf{l}^{\kappa-1}\overset{*}{\longrightarrow}\bf{l}^{\kappa}$ for all $\kappa\in\{1,\ldots,m\}$ (when $m\ne 0$).
\end{dfn}

We will show in the sequel that for well posed discretizations the sets $\P^{m}(Q_0)$, $m\in\{0,\ldots,\ell\}$ in Definition \ref{product:TS} are always nonempty and that there exists an outgoing transition in the product transition system from any  $\bf{l}\in \P^{m}(Q_0)$,  $m\in\{0,\ldots,\ell-1\}$. This property is a corollary of the following auxiliary results.

\begin{lemma} \label{lemma:path:in:Pk}
For any path $\bf{l}^0 \bf{l}^1\cdots\bf{l}^{m}$ of length $m\in\{0,\ldots,\ell\}$ originating from $\bf{l}^0$ in $TS^{\P}$ it holds $\bf{l}^{\kappa}\in\P^{\kappa}(Q_0)$ for all $\kappa\in\{0,\ldots,m\}$. 
\end{lemma}

\begin{proof}
The proof is carried out by induction on the length of the path and is based on the definitions of the operator $\P(\cdot)$ and the transitions in $TS^{\P}$. Notice that for $m=0$ it holds $\bf{l}^0\in\P^0(Q_0)=Q_0$ by the definition of a path. Assume now that the lemma is valid for certain $m\in\{0,\ldots,\ell-1\}$, i.e., for any path $\bf{l}^0 \bf{l}^1\cdots\bf{l}^{m}$ of length $m$ originating from $\bf{l}^0$ it holds $\bf{l}^{\kappa}\in\P^{\kappa}(Q_0)$ for all $\kappa\in\{0,\ldots,m\}$. In order to prove the lemma for $m=m+1$ it suffices to show that for any path $\bf{l}^0 \bf{l}^1\cdots\bf{l}^{m+1}$ of length $m+1$ originating from $\bf{l}^0$ it holds $\bf{l}^{\kappa}\in\P^{\kappa}(Q_0)$ for all $\kappa\in\{0,\ldots,m+1\}$. By invoking validity of the lemma for $m$ and the fact that $\bf{l}^0 \bf{l}^1\cdots\bf{l}^{m}$ is a path of length $m$, it suffices to show that $\bf{l}^{m+1}\in\P^{m+1}(Q_0)$. Indeed, since $\bf{l}^m \overset{*}{\longrightarrow}\bf{l}^{m+1}$, it holds that $\bf{l}^{m+1}\in  \P(\bf{l}^m)$. Since by induction $\bf{l}^m\in\P^m(Q_0)$, we obtain that $\P(\bf{l}^m)\subset\P(\P^
m(Q_0))=\P^{m+1}(Q_0)$ and thus, that $\bf{l}^{m+1}\in \P^{m+1}(Q_0)$, which completes the proof.
\end{proof}

\begin{prop} \label{propostion:Post:nonempty}
\textit{(i)} Let $m\in\{0,\ldots,\ell\}$  and assume that $\bf{l}^m\in  \P^m(Q_0)\ne\emptyset$. Then, there exists a path $\bf{l}^0 \bf{l}^1\cdots\bf{l}^{m}$ of length $m$ from $\bf{l}^0$ to $\bf{l}^m$ in $TS^{\P}$, such that $\bf{l}^{\kappa}\in\P^{\kappa}(Q_0)$ for all $\kappa\in\{0,\ldots,m\} $ and ${\rm pr}_i(\bf{l}^0)\cdots{\rm pr}_i(\bf{l}^{m-1})l^m_i$ is a strongly well posed transition sequence of order $m$ for each $i\in\N$ (see Definition \ref{dfn:strongly:well:posed:seq}).

\noindent \textit{(ii)} In addition, if the space-time discretization $\S-\delta t$ is well posed, then for each  $m\in\{0,\ldots,\ell\}$  it holds that $ \P^m(Q_0)\ne\emptyset$.
\end{prop}

\begin{proof}[Proof of (i)]
For the proof of Part (i), assume that without any loss of generality it holds $m\ne 0$ and notice that from \eqref{post:operator:property} and backward recursion we can select a path  $\bf{l}^0 \bf{l}^1\cdots\bf{l}^{m}$  of length $m$ from $\bf{l}^0$ to $\bf{l}^m$ in $TS^{\P}$, implying that $\bf{l}^0\in Q_0$. Then, it follows from Lemma~\ref{lemma:path:in:Pk} that $\bf{l}^{\kappa}\in\P^{\kappa}(Q_0)$ for all $\kappa\in\{0,\ldots,m\}$. In addition, for all $\kappa\in\{1,\ldots,m\}$ it follows from the fact that $\bf{l}^{\kappa-1}\overset{*}{\longrightarrow}\bf{l}^{\kappa}$ and the definition of $\P(\cdot)$ in \eqref{post:operator:def}, that $l_i^{\kappa-1}\overset{{\rm pr}_i(\bf{l}^{\kappa-1})}{\longrightarrow_i}l_i^{\kappa}$ for each $i\in\N$. Consequently, we get from Definition \ref{dfn:strongly:well:posed:seq} that ${\rm pr}_i(\bf{l}^0)\cdots{\rm pr}_i(\bf{l}^{m-1})l^m_i$ is a strongly well posed transition sequence of order $m$ for each $i\in\N$.  

\noindent \textit{Proof of (ii).} The proof of (ii) is carried out by induction on $m$ and exploits the results of Part~(i) and Proposition  \ref{proposition:implementation:of:strwellposed}(ii).
For $m=0$, the result follows directly from the fact that $\P^0(Q_0)=Q_0\ne\emptyset$. 
For the general case, assume that $\P^m(Q_0)\ne\emptyset$ for certain $m\in\{0,\ldots,\ell-1\}$ and let $\bf{l}^m\in \P^m(Q_0)$. Then, it follows from Part (i) that there exists $\bf{l}^0\bf{l}^1\cdots\bf{l}^m$ such that ${\rm pr}_i(\bf{l}^0)\cdots{\rm pr}_i(\bf{l}^{m-1})l^m_i$ is a strongly well posed transition sequence of order $m$ for each $i\in\N$. Hence, we obtain from Proposition~\ref{proposition:implementation:of:strwellposed}(ii) that ${\rm Post}_i(l_i^m;{\rm pr}_i(\bf{l}^m))\ne\emptyset$ for all $i\in\N$. Consequently, we deduce from \eqref{post:operator:def} that $\P(\bf{l}^m)\ne\emptyset$ and thus, since $\bf{l}^m\in \P^m(Q_0)$, that $\P^{m+1}(Q_0)=\P(\P^m(Q_0))\supset \P(\bf{l}^m)\ne\emptyset$.    
\end{proof}

\begin{corollary} \label{corollary:Post:nonempty}
Assume that the space-time discretization $\S-\delta t$ is well posed. Then, for each $m\in\{0,\ldots,\ell-1\}$ and $\bf{l}\in\P^m(Q_0)(\ne\emptyset)$ it holds ${\rm Post}(\bf{l})=\P(\bf{l})\ne \emptyset$.  
\end{corollary}

\begin{proof}
Let  $m\in\{0,\ldots,\ell-1\}$ and recall that by Proposition~\ref{propostion:Post:nonempty}(ii) it holds that $\P^m(Q_0)\ne\emptyset$. Then, given $\bf{l}\in\P^m(Q_0)$ it follows from Proposition~\ref{propostion:Post:nonempty}(i) that there exists  $\bf{l}^0\in Q_0$ and a path $\bf{l}^0 \bf{l}^1\cdots\bf{l}^m$ of length $m$ from $\bf{l}^0$ to $\bf{l}^m=\bf{l}$ in $TS^{\P}$, such that $\bf{l}^{\kappa}\in\P^{\kappa}(Q_0)$ for all $\kappa\in\{0,\ldots,m\}$ and ${\rm pr}_i(\bf{l}^0)\cdots{\rm pr}_i(\bf{l}^{m-1})l^m$ is a strongly well posed transition sequence of order $m$ for each $i\in\N$. Hence, we obtain from Proposition~\ref{proposition:implementation:of:strwellposed}(ii) that ${\rm Post}_i(l_i^m;{\rm pr}_i(\bf{l}^m))\ne\emptyset$ for all $i\in\N$, which by virtue of \eqref{post:operator:def} implies that $\P(\bf{l}^m)\ne\emptyset$. Also, since $\bf{l}^m\in\P^m(Q_0)$ we have for all $\bf{l}'\in\P(\bf{l}^m)$ that $\bf{l}^m \overset{*}{\longrightarrow}\bf{l}'$, implying that  $\bf{l}'\in{\rm Post}(\bf{l}^m)$ and hence, that  $\P(\bf{l}^m)\subset{\rm Post}(\bf{l}^m)$. Finally, from the definition of the transitions in $TS^{\P}$ we have that $\bf{l}'\in{\rm Post}(\bf{l}^m)$, or equivalently, that $\bf{l}^m \overset{*}{\longrightarrow}\bf{l}'$ only if $\bf{l}'\in\P(\bf{l}^m)$, implying that also ${\rm Post}(\bf{l}^m)\subset\P(\bf{l}^m)$. The proof is now complete.
\end{proof}

The proposition below constitutes our main result in this section and guarantees the existence of paths of length $m$ for any $m\in\{0,\ldots,\ell\}$ originating from certain $\bf{l}^0\in Q_0$ in $TS^{\P}$. Additionally, it is shown that any such path can be realized by a sampled trajectory of the continuous time system \eqref{single:integrator} initiated from $\bf{X}_{0}$ over the subinterval $[0,m\delta t]$ of the time horizon $[0,T]=[0,\ell\delta t]$. 

\begin{prop} \label{proposition:sampled:traj:consistency}
Consider cell decompositions $\S_i=\{S_l^i\}_{l\in\I_i}$ of $\R_i([0,T])$, $i\in\N$, their product $\S$, a time step $\delta t<\tau$ with $T=\ell\delta t$, nonempty subsets $W_i$, $i\in\N$ of $\Rat{n}$ and assume that each $\S_i$ is compliant with  $\R_i([0,T-\delta t])$ and that the space time discretization $\S-\delta t$ is well posed. Then, the following hold.

\noindent \textit{(i)} For any $m\in\{0,\ldots,\ell\}$ there exists a path $\bf{l}^0 \bf{l}^1\cdots\bf{l}^{m}$ of length $m$ originating from $\bf{l}^0$ in the product transition system $TS^{\P}$.

\noindent \textit{(ii)} For any path $\bf{l}^0 \bf{l}^1\cdots\bf{l}^{m}$ of length $m$ originating from $\bf{l}^0$ in $TS^{\P}$, there exists an input $v\in\mathcal{U}$ such that each component $x_i(\cdot,\bf{X}_0;v)$ of the solution of \eqref{single:integrator} satisfies $x_i(\kappa\delta t,\bf{X}_0;v)\in S^i_{l_i^{\kappa}}$, for all $\kappa\in\{0,\ldots,m\}$.
\end{prop}

\begin{proof}[Proof of (i)]
The proof of Part (i) is a direct consequence of Corollary \ref{corollary:Post:nonempty}.

\noindent \textit{Proof of (ii).} Let $\bf{l}^0 \bf{l}^1\cdots\bf{l}^{m}$  be a path of length $m$ originating from $\bf{l}^0$ and assume that without any loss of generality it holds $m\ne 0$. Then, it follows from the definition of the path that $\bf{l}^{\kappa-1}\overset{*}{\longrightarrow}\bf{l}^{\kappa}$ for all $\kappa\in\{1,\ldots,m\}$ and thus, that $\bf{l}^{\kappa}\in\P(\bf{l}^{\kappa-1})$ for all $\kappa\in\{1,\ldots,m\}$, implying that $l_i^{\kappa-1}\overset{{\rm pr}_i(\bf{l}^{\kappa-1})}{\longrightarrow_i}l_i^{\kappa}$ for each $\kappa\in\{1,\ldots,m\}$ and $i\in\N$. Hence, by recalling that $\bf{l}^0\in Q_0$, it follows that ${\rm pr}_i(\bf{l}^0)\cdots{\rm pr}_i(\bf{l}^{m-1})l^m_i$ is a strongly well posed transition sequence of order $m$ for each $i\in\N$, and we  
deduce from Proposition \ref{proposition:implementation:of:strwellposed}(i) that there exists an input $v\in\mathcal{U}$ such that each component $x_i(\cdot,\bf{X}_0;v)$ of the solution of \eqref{single:integrator} satisfies $x_i(\kappa\delta t,\bf{X}_0;v)\in S^i_{l_i^{\kappa}}$, for all $\kappa\in\{0,\ldots,m\}$. The proof is now complete. 
\end{proof}

\section{Design of the Hybrid Control Laws}

Consider again system  \eqref{single:integrator}. According to Definition \ref{well:posed:discretization}, establishment of a well posed discretization is based on the design of appropriate feedback laws which guarantee well posed transitions for all agents and their possible cell configurations, based on the auxiliary system with disturbances \eqref{system:disturbances}. We proceed by defining the control laws that are exploited in order to derive well posed discretizations. Consider for each agent $i$ a cell decomposition $\{S_{l}^i\}_{l\in\I_i}$ of its reachable set $\R_i([0,T])$ and a time step $\delta t$. We define the diameter $d_{\max}(i)$ of each cell decomposition $\{S_l^i\}_{l\in\I_i}$ as
\begin{align} 
d_{\max}(i):=\inf\{R>0:\forall l\in\I_i,\exists x\in S_l^i,S_l^i\subset B(x;\tfrac{R}{2})\}  \label{dmax:dfn}
\end{align}

\noindent and select a reference point $x_{l_i,G}$ for every cell $S_{l_i} ^i$, satisfying 
\begin{equation} \label{reference:points}
|x_{l_i,G}-x|\le\frac{d_{\max}(i)}{2},\forall x\in S_{l_i}^i,l_i\in\I_i,i\in\N.
\end{equation}

\noindent  For each agent $i$ and cell configuration $\bf{l}_i$ of $i$, we define the family of feedback laws  $k_{i,\bf{l}_i}:[0,\infty)\times \Rat{(N_i+1)n}\to\Rat{n}$ parameterized by $x_{i0}\in S_{l_i}^i$ and $w_i\in W_i$ as $k_{i,\bf{l}_i}(t,x_i,\bf{x}_j;x_{i0},w_i):=k_{i,\bf{l}_i,1}(t,x_i,\bf{x}_j)+ k_{i,\bf{l}_i,2}(x_{i0})+k_{i,\bf{l}_i,3}(w_i)$, where
\begin{align}
W_i:= & B(v_{\max}(i))\subset\Rat{n}, \label{set:W} \\
k_{i,\bf{l}_i,1}(t,x_{i},\bf{x}_j):= & g_i(\chi_i(t),\bf{x}_{l_j,G})-g_i(x_i,\bf{x}_j), \label{feedback:ki1} \\
k_{i,\bf{l}_i,2}(x_{i0}):= & \frac{1}{\delta t}(x_{l_i,G}-x_{i0}), \label{feedback:ki2} \\
k_{i,\bf{l}_i,3}(w_i):= & \lambda(i)w_i,\label{feedback:ki3:plan} \\
t\in [0,\infty),(x_i,\bf{x}_j)\in & \Rat{(N_i+1)n},x_{i0}\in S_{l_i}^i,w_i\in W_i. \nonumber
\end{align}

\noindent The function  $\chi_i(\cdot)$ in \eqref{feedback:ki1} is defined for all $t\ge 0$ through the solution of the initial value problem 
\begin{equation} \label{reference:trajectory}
\dot{\chi}_i=g_i(\chi_i,\bf{x}_{l_j,G}), \chi_i(0)=x_{l_i,G} 
\end{equation}

\noindent with $g_i(\cdot)$ as given in \eqref{function:g}. Recall that since $g_i(\cdot)$ is globally Lipschitz, the trajectory $\chi_i(\cdot)$ is defined for all $t\ge 0$. We also note that the feedback laws $k_{i,\bf{l}_i}(\cdot)$ depend on the cell of agent $i$ and specifically on its cell configuration $\bf{l}_i$, through the reference points $\bf{x}_{l_j,G}$ and  $x_{l_i,G}$ in \eqref{feedback:ki1} and \eqref{feedback:ki2}, and the trajectory $\chi_i(\cdot)$ in \eqref{feedback:ki1} as provided by \eqref{reference:trajectory}. The parameter $\lambda(i)$ stands for the part of the free input that can be further exploited for motion planning. In particular, for each $w_i\in W$ in \eqref{set:W}, the vector $\lambda(i) w_i$ provides the ``velocity" of a motion that we superpose to the reference trajectory $\chi_i(\cdot)$ of agent $i$ over $[0,\delta t]$. The latter allows the agent to reach all points inside a ball with center the position of the reference trajectory at time $\delta t$ by following the curve $\bar{x}_i(t):=\chi_i(t)+\lambda(i)w_it$, as depicted in Fig. \ref{fig:controllers} below. This ball has radius
\begin{equation} \label{distance:r}
r_i:=\int_0^{\delta t}\lambda(i)dsv_{\max}(i)=\lambda(i)\delta tv_{\max}(i),
\end{equation}

\noindent namely, the distance that the agent can cross in  time $\delta t$ by exploiting $k_{i,\bf{l}_i,3}(\cdot)$, which corresponds to the part of the free input that is selected  for reachability purposes. Hence, it is possible to perform a well posed transition to any cell which has a nonempty intersection with $B(\chi_i(\delta t);r_i)$.

\begin{figure}[H]
\begin{center}

\begin{tikzpicture}[scale=1]

\draw[dashed, color=red,thick] (3,0.5) circle (1cm);

\draw[color=gray] (0.1,0.8) -- (0.2,1.4) -- (1,1.5) -- (1.8,0.8);
\draw[color=gray] (1,1.5) -- (2.1,1.7);
\draw[color=gray] (0.3,-0.7) -- (0.8,0.1) -- (1.8,0.8)-- (1.6,-0.5) -- (0.3,-0.7);
\draw[color=gray] (0.1,0.8) -- (1.8,0.8);
\draw[color=green!50!black,thick] (-0.3,-0.5) -- (0.3,-0.7) -- (0.8,0.1) -- (0.1,0.8)-- (-0.3,-0.5);
\draw[color=cyan!70!black,thick] (1.8,0.8) -- (2.8,-0.7)-- (1.6,-0.5) -- (1.8,0.8);
\draw[color=cyan!70!black,thick] (1.8,0.8)-- (2.8,-0.7)-- (3.3,0.3) -- (3,1) -- (1.8,0.8);
\draw[color=cyan!70!black,thick] (2.8,-0.7)-- (4,-0.6) -- (4.1,0.4) -- (3.3,0.3) ;
\draw[color=cyan!70!black,thick] (1.8,0.8) -- (2.1,1.7) -- (3.2,1.7);
\draw[color=cyan!70!black,thick] (3,1) -- (3.2,1.7) -- (4.2,1.2) -- (4.1,0.4);

\fill[black] (3,0.5) circle (2pt);
\fill[black] (3.3,-0.2) circle (2pt);

\draw[black, dashed] (0.2,0.05) circle (0.7566cm);


\coordinate [label=left:$\textcolor{green!50!black}{S_{l_i}^i}$] (A) at (-0.3,-0.5);
\coordinate [label=right:$\textcolor{cyan!70!black}{S_{l_i'}^i}$] (A) at (3,-1);

\draw[dashed,->,>=stealth'] (3.2,0.5) -- (4.5,0.5) node[right] {$\chi_{i}(\delta t)$};
\draw[dashed,->,>=stealth'] (3.5,-0.2) -- (4.5,-0.2) node[right] {$x_{i}(\delta t)=x$};

\draw[color=red,dashed,thick] (6,0.5) -- (6.5,0.5);
\coordinate [label=right:$B(\chi_{i}(\delta t);r_i)$] (A) at (6.5,0.5);

\draw[color=blue,dashed,thick] (6,-0.7) -- (6.5,-0.7);
\coordinate [label=right:$\chi_{i}(\delta t)+t\lambda(i)w_i$] (A) at (6.5,-0.7);


\draw[color=red,thick] (6,1.2) -- (6.5,1.2);
\coordinate [label=right:$x_{i}(t)$] (A) at (6.5,1.2);

\draw[color=olive!50!black,very thick,densely dotted] (6,1.9) -- (6.5,1.9);
\coordinate [label=right:$\bar{x}_{i}(t)$] (A) at (6.5,1.9);

\draw[color=blue,dashed,->,thick,>=stealth'] (3,0.5) -- (3.3,-0.2);

\draw[color=blue,->,thick,>=stealth'] (0.2,0.05)  .. controls (1.4,0.7) .. (3,0.5);
\draw[color=olive!50!black,->,very thick,>=stealth',densely dotted] (0.2,0.05)  .. controls (1.6,0.5) .. (3.3,-0.2);
\draw[color=red,->,thick,>=stealth'] (0.2,-0.6) .. controls (1.6,0.3) .. (3.3,-0.2);

\fill[black] (0.2,-0.6) circle (1.5pt) node[below] {$x_{i0}$};
\fill[black] (0.2,0.05) circle (2pt) node[below]{$x_{l_i,G}$};

\end{tikzpicture}
\vspace{-0.4cm}

\end{center}
\caption{Illustration of the reference trajectory and reachability capabilities of the control laws.} \label{fig:controllers}
\end{figure}
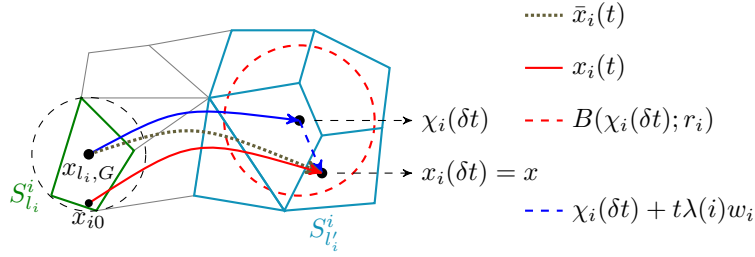

\section{Well Posed Space-Time Discretizations}

In this section, we exploit the controllers introduced in Section 4 to provide sufficient conditions for well posed space-time discretizations. By exploiting the result of Proposition \ref{proposition:sampled:traj:consistency} this  framework can be applied for motion planning, by specifying different transition possibilities for each agent through modifying its controller appropriately. As in the previous section, we consider the system \eqref{single:integrator}, cell decompositions $\mathcal{S}_i=\{S_l^i\}_{l\in\I_i}$ of the reachable sets $\R_i([0,T])$, $i\in\N$, a time step $\delta t$ and a selection of reference points $x_{l_i,G}$, $l\in\I_i$, $i\in\N$, as in \eqref{reference:points}.

Since the discrete space is updated after the end of the time horizon and the space discretization of each agent is affected by the local (in time) properties of its dynamics, it provides a convenient setting for the consideration of different diameters for the decomposition of each agent. Thus, we will additionally introduce certain design constraint relating the diameters of neighboring decompositions. In particular for each agent's neighbors we impose the restriction that the diameters of their decompositions satisfy  
\begin{equation} \label{diameters:restrictions}
d_{\max}(j)\le \mu(j,i) d_{\max}(i). 
\end{equation}

\noindent Note that for these restrictions to be meaningful we also need to impose the condition that along each cycle in the communication graph their product is greater or equal than one, i.e., 
\begin{equation} \label{diameters:restrictions:consistency}
\mu(i_0,i_1)\mu(i_1,i_2)\cdots\mu(i_{m-1},i_m)\ge 1,\;\textup{for all cycles} \;i_0i_1\cdots i_m\;{\rm with}\;i_0=i_m \;{\rm in}\; \mathcal{G},
\end{equation}

\noindent which is always satisfied if we select $\mu(j,i)=1$ for all $i\in\N$ and $j\in\N_i$. For the acceptable values of the discretizations, it is also convenient to define the following local network parameters for each agent.
\begin{align}
\mbf{\mu}(i):= & \left(\sum_{j\in\N_i}\mu(j,i)^2\right)^{\frac{1}{2}}, \\
\bf{M}(i):= &  \left(\sum_{j\in\N_i}(M(j)+v_{\max}(j))^2\right)^{\frac{1}{2}}.
\end{align}


\noindent For each $i\in\N$ and cell configuration $\bf{l}_i\in\bfI_i$ of $i$, consider the family of feedback laws given in \eqref{feedback:ki1}, \eqref{feedback:ki2}, \eqref{feedback:ki3:plan}, and parameterized by $x_{i0}\in S_{l_i}^i$ and $w_i\in W_i$. As in the previous section, $\chi_i(\cdot)$ is the reference trajectory generated by the initial value problem \eqref{reference:trajectory} and the parameter $\lambda(i)$ provides for each agent a measure of the control that is exploited for reachability purposes. We proceed by providing the desired sufficient conditions for well posed space-time discretizations and their transition capabilities.

\begin{thm}\label{discretizations:for:planning}
Consider cell decompositions $\mathcal{S}_i=\{S_l^i\}_{l\in\I_i}$ of the reachable sets $\R_i([0,T])$ with diameters $d_{\max}(i)$, their product $\S$, a time step $\delta t$, the constant $r_i$ defined in \eqref{distance:r}, the parameters $\lambda(i)\in(0,1)$ and assume that each $\S_i$ is compliant with $\R_i([0,T-\delta t])$. We also assume that $d_{\max}(i)$ and $\delta t$ satisfy \eqref{diameters:restrictions}, \eqref{diameters:restrictions:consistency}, $\ell \delta t=T$ for certain $\ell\in\mathbb{N}$ and the following additional restrictions:

\begin{align}
\delta t \in & \left(0,\frac{(1-\lambda(i))v_{\max}(i)}{L_1(i)\bf{M}(i)+L_2(i)\lambda(i)v_{\max}(i)}\right), \label{deltat:interval} \\
d_{\max} \in & \left(0,\min\left\lbrace \frac{2(1-\lambda(i))v_{\max}(i)\delta t}{1+(L_1(i)\mu(i)+L_2(i))\delta t},\right.\right. \nonumber \\
& \left.\left.\frac{2(1-\lambda(i))v_{\max}(i)\delta t-2(L_1(i)\bf{M}(i)+L_2(i)\lambda(i)v_{\max}(i)) \delta t^2}{1+L_1(i)\mbf{\mu}(i)\delta t} \right\rbrace\right), \label{dmax:interval}
\end{align}

\noindent with $L_1(i)$, $L_2(i)$ and $v_{\max}(i)$ as given in \eqref{dynamics:bound1}, \eqref{dynamics:bound2} and \eqref{input:bound}, respectively,  and $\mbf{\mu}(i)$, $\bf{M}(i)$ as defined above. Then, the space-time discretization is well posed for the multi-agent system \eqref{single:integrator}. In particular, the following hold. 

\noindent \textit{(i)} For each agent $i\in\N$, cell decomposition $\S_i'=\{S_{l}^i\}_{l\in\I_i'}$ of $\R_i^{c_i(\tau)}([0,T-\tau])$ with $\S_i'\supset\S_i$, $\I_i'\supset\I_i$ and compliant with $\R_i([0,T])$, and cell configuration $\bf{l}_i$ of $i$ with $S_{l_i}^i\subset\R_i([0,T-\delta t])$ it holds
\begin{align} 
 B(\chi_i(\delta t);r_i) & \subset \R_i^{c_i(\tau)}([0,T-\tau]), \label{reachball:inclusion}\\
l_i\overset{\bf{l}_i}{\longrightarrow}l_i'\;\textup{is well posed}\; \forall l_i' & \in \{l\in\I_i':S_l^i\cap B(\chi_i(\delta t);r_i)\ne\emptyset\}, \label{planning:condition}
\end{align}

\noindent with the reference trajectory $\chi_i(\cdot)$ as given by \eqref{reference:trajectory} and $r_i$ as defined in \eqref{distance:r}.

\noindent \textit{(ii)}  For each agent $i\in\N$, cell configuration $\bf{l}_i$ of $i$, integer $m\in\{0,\ldots,\ell-1\}$ and input $v\in\U$ such that each component $x_{\kappa}(\cdot,\bf{X}_0;v)$, $\kappa\in\N_i\cup\{i\}$ of the solution of \eqref{single:integrator} satisfies  $x_{\kappa}(m\delta t,\bf{X}_0;v)\in S_{l_{\kappa}} ^{\kappa}$, it holds ${\rm Post}(l_i,\bf{l}_i)\supset \{l\in\I_i':S_l^i\cap B(\chi_i(\delta t);r_i)\ne\emptyset\}$.
\end{thm}

\begin{proof}[Proof of (i)]
For the proof, pick $i\in\N$, a cell decomposition  $\S_i'=\{S_{l}^i\}_{l\in\I_i'}$ of $\R_i^{c_i(\tau)}([0,T-\tau])$ with $\S_i'\supset\S_i$, $\I_i'\supset\I_i$ and compliant with $\R_i^{c_i(\tau)}([0,T-\tau])$, and a cell configuration $\bf{l}_i$ of $i$ with $S_{l_i}^i\subset\R_i([0,T-\delta t])$. Also consider the reference trajectory $\chi_i(\cdot)$ in \eqref{reference:trajectory} which is defined for all $t\ge 0$.

We first show that \eqref{reachball:inclusion} is satisfied. Indeed, let any $x\in B(\chi_i(\delta t);r_i)$. Then, it follows from \eqref{function:g} and \eqref{reference:trajectory} that 
\begin{equation} \label{reference:tr:bound}
|\chi_i(\delta t)-x_{l_i,G}|\le\int_0^{\delta t}|g_i(\chi_i(t),\bf{x}_{l_j,G})|dt\le M(i)\delta t.
\end{equation}

\noindent In addition, we have from \eqref{reachset:includedin:sigma:interval:inflated} applied with $\sigma=\tau-\delta t$ that 
\begin{equation} \label{inflation:of:tau:min:deltat}
\R_i^{c_i(\tau-\delta t)}([0,T-\tau])\supset \R_i([0,T-\delta t])
\end{equation}

\noindent and from the fact that $x_{l_i,G}\in S_{l_i}^i\subset \R_i([0,T-\delta t])$ that 
\begin{equation} \label{xig:in:inflation}
x_{l_i,G}\in \R_i^{c_i(\tau-\delta t)}([0,T-\tau]).
\end{equation}

\noindent Furthermore, since $x\in B(\chi_i(\delta t);r_i)$, we obtain from \eqref{distance:r} that $|\chi_i(\delta t)-x|\le\lambda(i)v_{\max}(i)\delta t$, which in conjunction with \eqref{reference:tr:bound}, \eqref{xig:in:inflation}, \eqref{constant:ci} and \eqref{Rti:over:interval:inflaed:semigroup} applied with $c=c_i(\tau-\delta t)$ and $\bar{c}=c_i(\delta t)$ implies that
\begin{align} \label{x:in:biger:reach}
x & \in  \R_i^{c_i(\tau-\delta t)}([0,T-\tau])+B(c_i(\delta t)) \nonumber \\
& = \R_i^{c_i(\tau-\delta t)+c_i(\delta t)}([0,T-\tau])= \R_i^{c_i(\tau)}([0,T-\tau]),
\end{align}

\noindent as desired.

For the derivation of \eqref{planning:condition}, namely, that each transition $l_i\overset{\bf{l}_i}{\longrightarrow}l_i'$, with $S_{l_i'}^i\cap B(\chi_i(\delta t);r_i)\ne\emptyset$ is well posed, it suffices to show that for each $x\in B(\chi_i(\delta t);r_i)$ and $l_i'\in\I_i'$ such that $x\in S_{l_i'}^i$, the transition $l_i\overset{\bf{l}_i}{\longrightarrow}l_i'$ is well posed. Thus, for each $x\in B(\chi_i(\delta t);r_i)$ and $l_i'\in\I_i'$ such that $x\in S_{l_i'}^i$, we need according to Definition \ref{well:posed:discretization}(i) to find a feedback law \eqref{feedback:for:i} satisfying Property (P) and a vector $w_i\in W_i$, in such a way that the Consistency Condition is fulfilled. Let $x\in B(\chi_i(\delta t);r_i)$ and define
\begin{equation}  \label{vector:wi}
w_i:= \frac{x-\chi_i(\delta t)}{\lambda(i)\delta t},
\end{equation}

\noindent with $\lambda(i)$ as in the statement of the theorem. Then, it follows from \eqref{distance:r} that $|w_i|\le \frac{r_i}{\lambda(i)\delta t}\le v_{\max}(i)$ and hence, by virtue of \eqref{set:W} that $w_i\in W_i$. We now select the feedback law $k_{i,\bf{l}_i}(\cdot)$ as given by \eqref{feedback:ki1}, \eqref{feedback:ki2}, \eqref{feedback:ki3:plan} and with $w_i$ as defined in \eqref{vector:wi}, and we will show that for all $l_i'\in\I_i'$ such that $x\in S_{l_i'}^i$ the Consistency Condition is satisfied. Notice first that $k_{i,\bf{l}_i}(\cdot)$ satisfies Property (P) . In order to show the Consistency Condition we pick  $x_{i0}\in S_{l_i}^i$, $x_i':=x$ with $x$ as selected above and prove that the solution $x_i(\cdot)$ of \eqref{system:disturbances} with $v_i=k_{i,\bf{l}_i}(t,x_i,\bf{d}_j;x_{i0},w_i)$ satisfies \eqref{xi:consistency:bounds}, \eqref{xi:in:finalcell} and \eqref{controler:consistency}, for any continuous $d_{j_1},\ldots,d_{j_{N_i}}:\RgeO\to\Rat{n}$ that satisfy \eqref{disturbance:bounds}. We break the subsequent proof in the following steps.

\noindent \textit{STEP 1:  Proof of \eqref{xi:consistency:bounds} and \eqref{xi:in:finalcell}.} By taking into account \eqref{system:disturbances} and \eqref{feedback:ki1}-\eqref{feedback:ki3:plan} we obtain for any continuous $d_{j_1},\ldots,d_{j_{N_i}}:\RgeO\to\Rat{n}$ the solution $x_i(\cdot)$ of \eqref{system:disturbances} with $v_i=k_{i,\bf{l}_i}$ as

\begin{align}
x_i(t) & =x_{i0}+\int_0^t(g_i(x_i(s),\bf{d}_j(s))+k_{i,\bf{l}_i}(s,x_{i}(s),\bf{d}_j(s);x_{i0},w_i))ds  \nonumber \\
& = x_{i0}+\int_0^t\left(g_i(\chi_i(s),\bf{x}_{l_j,G})+\frac{1}{\delta t}(x_{l_i,G}-x_{i0})+\lambda(i) w_i\right)ds  \nonumber \\
& = x_{i0}+ \frac{t}{\delta t}(x_{l_i,G}-x_{i0})+\lambda(i) w_it+\int_0^tg_i(\chi_i(s),\bf{x}_{l_j,G})ds,t\ge 0. \label{solution:expansion}
\end{align}

\noindent Hence, we deduce from \eqref{function:g}, \eqref{reference:points}, \eqref{solution:expansion} and the fact that from \eqref{dmax:interval} it holds $d_{\max}(i)< 2(1-\lambda(i))v_{\max}(i)\delta t$, that
\begin{align} \label{solution:vs:reftraj}
|x_i(t)-x_{i0}| \le & \left|\frac{t}{\delta t}(x_{i0}-x_{l_i,G})\right|+\lambda(i)|w_i|t+M(i)t \nonumber \\
\le & \frac{d_{\max}(i)t}{2\delta t}+(\lambda(i)v_{\max}(i)+M(i))t < \frac{2(1-\lambda(i))v_{\max}(i)\delta t}{2\delta t}t \nonumber \\
+ & (\lambda(i)v_{\max}(i)+M(i))t=(v_{\max}(i)+M(i))t,\forall t\in (0,\delta t].
\end{align}

\noindent which establishes validity of \eqref{xi:consistency:bounds}. Furthermore, we get from \eqref{solution:expansion} and \eqref{reference:trajectory} that 
\begin{align}
x_i(t) & = \frac{\delta t-t}{\delta t}(x_{i0}-x_{l_i,G})+\lambda(i)w_it+x_{l_i,G}+\int_0^tg_i(\chi_i(s),\bf{x}_{l_j,G})ds \nonumber \\
& = \frac{\delta t-t}{\delta t}(x_{i0}-x_{l_i,G})+\lambda(i)w_it+\chi_i(t),t\ge 0, \label{solution:expression}
\end{align}

\noindent which implies that $x_i(\delta t)=\chi_i(\delta t)+\delta t\lambda(i) w_i=x=x_i'$ and thus, \eqref{xi:in:finalcell} also holds. 

\noindent \textit{STEP 2: Proof of the fact that} 
\begin{equation} \label{time:bart:thm}
x_i(t)\in\R_i^{c_i(\tau)}([0,T-\tau]), \forall t\in[0,\delta t].
\end{equation}

\noindent Notice first that from \eqref{solution:vs:reftraj} it follows that  
\begin{equation} \label{solution:vs:ref:point}
|x_i(t)-x_{i0}|< (v_{\max}(i)+M(i))\delta t,\forall t\in[0,\delta t].
\end{equation}

\noindent In addition, we have that $x_{i0}\in  S_{l_i}^i\subset R_i([0,T-\delta t])$ and from \eqref{inflation:of:tau:min:deltat} that  $x_{i0}\in  R_i^{c_i(\tau-\delta t)}([0,T-\tau])$. Thus, we obtain from the latter, \eqref{solution:vs:ref:point} and the same arguments that were applied for the derivation of \eqref{x:in:biger:reach} that \eqref{time:bart:thm} is fulfilled.

\noindent \textit{STEP 3: Estimation of bounds on $k_{i,\bf{l}_i,1}(\cdot)$, $k_{i,\bf{l}_i,2}(\cdot)$ and $k_{i,\bf{l}_i,3}(\cdot)$ along the solution $x_i(\cdot)$ of \eqref{system:disturbances} with $v_i=k_{i,\bf{l}_i}$ and $d_{j_1},\ldots,d_{j_{N_i}}$ satisfying  \eqref{disturbance:bounds}.}  We first show that
\begin{align}
|k_{i,\bf{l}_i,1}(t,x_{i}(t),\bf{d}_j(t))| & \le L_{1}(i)\left(\mbf{\mu}(i)\frac{d_{\max}(i)}{2}+\bf{M}(i)t\right) \nonumber \\
& +L_{2}(i)\left(\frac{(\delta t-t)d_{\max}(i)}{2\delta t}+\lambda(i)v_{\max}(i)t\right), \forall t\in[0,\delta t]. \label{ki1:bound}
\end{align}

\noindent Indeed, notice that by virtue of \eqref{feedback:ki1} we have
\begin{equation}  \label{ki1:equiv}
k_{i,\bf{l}_i,1}(t,x_{i}(t),\bf{d}_j(t))=[g_i(\chi_i(t),\bf{x}_{l_j,G})-g_i(x_i(t),\bf{x}_{l_j,G})]+[g_i(x_i(t),\bf{x}_{l_j,G})-g_i(x_i(t),\bf{d}_j(t))].
\end{equation}

\noindent For the second difference on the right hand side of \eqref{ki1:equiv}, we obtain from \eqref{dynamics:bound1},  \eqref{disturbance:bounds}, \eqref{reference:points}, \eqref{diameters:restrictions},  \eqref{time:bart:thm} and the Cauchy Schwartz inequality that
\begin{align*}
& |g_i(x_i(t),\bf{x}_{l_j,G}) -  g_i(x_i(t),\bf{d}_j(t))| \le L_{1}(i)|(d_{j_{1}}(t)-x_{l_{j_1},G},\ldots,d_{j_{N_i}}(t)-x_{l_{j_N},G})| \\
& \le  L_{1}(i)\left(\sum_{\kappa\in\N_i}\left(\frac{d_{\max}(\kappa)}{2}+(M(\kappa)+v_{\max}(\kappa))t\right)^{2}\right)^{\frac{1}{2}} \\
& \le  L_{1}(i)\left(\sum_{\kappa\in\N_i}\left(\mu(\kappa,i)\frac{d_{\max}(i)}{2}+(M(\kappa)+v_{\max}(\kappa))t\right)^{2}\right)^{\frac{1}{2}} \\ 
& = L_{1}(i)\left(\frac{d_{\max}(i)^2}{4}\sum_{\kappa\in\N_i}\mu(\kappa,i)^2+2\frac{d_{\max}(i)}{2}t\sum_{\kappa\in\N_i}\mu(\kappa,i)(M(\kappa)+v_{\max}(\kappa))\right. \\
& \left.\hspace{20em} +t^2\sum_{\kappa\in\N_i}(M(\kappa)+v_{\max}(\kappa))^2\right)^{\frac{1}{2}}
\end{align*}

\begin{align*}
 &\le L_{1}(i)\left(\frac{d_{\max}(i)^2}{4}\sum_{\kappa\in\N_i}\mu(\kappa,i)^2+2\frac{d_{\max}(i)}{2}t\left(\sum_{\kappa\in\N_i}\mu(\kappa,i)^2\right)^{\frac{1}{2}}
\left(\sum_{\kappa\in\N_i}(M(\kappa)+v_{\max}(\kappa))^2\right)^{\frac{1}{2}}\right. \\
&\left.\hspace{26em} +t^2\sum_{\kappa\in\N_i}(M(\kappa)+v_{\max}(\kappa))^2\right)^{\frac{1}{2}} \\
& = L_{1}(i)\left(\left(\left(\sum_{\kappa\in\N_i}\mu(\kappa,i)^2\right)^{\frac{1}{2}}\frac{d_{\max}(i)}{2}+
\left(\sum_{\kappa\in\N_i}(M(\kappa)+v_{\max}(\kappa))^2\right)^{\frac{1}{2}}t\right)^2\right)^{\frac{1}{2}} \\
& = L_{1}(i)\left(\mbf{\mu}(i)\frac{d_{\max}(i)}{2}+\bf{M}(i)t\right),\forall t\in[0,\delta t].
\end{align*}

\noindent For the other difference in \eqref{ki1:equiv}, it follows from \eqref{dynamics:bound2}, \eqref{time:bart:thm}, \eqref{reference:points} and \eqref{solution:expression} that
\begin{align*}
|g_i(x_i(t),\bf{x}_{l_j,G})-g_i(\chi_i(t),\bf{x}_{l_j,G})| \le & L_2(i)\left|\left(\chi_i(t)+\left(\frac{\delta t-t}{\delta t}\right)(x_{i0}-x_{l_i,G})+\lambda(i) w_it\right)-\chi_i(t)\right| \\
\le & L_{2}(i)\left(\frac{(\delta t-t)d_{\max}(i)}{2\delta t}+\lambda(i)v_{\max}(i)t\right), \forall t\in[0,\delta t].
\end{align*}

\noindent Hence, it follows from the evaluated bounds on the differences of the right hand side of \eqref{ki1:equiv} that \eqref{ki1:bound} holds. Next, by recalling that $x_{l_i,G}$ satisfies \eqref{reference:points}, it follows directly from \eqref{feedback:ki2} that
\begin{equation} \label{ki2:bound}
|k_{i,\bf{l}_i,2}(x_{i0})|=\frac{1}{\delta t}|x_{i0}-x_{l_i,G}|\le \frac{d_{\max}(i)}{2\delta t},\forall x_{i0}\in S_{l_i}^i.
\end{equation}

\noindent Finally, for $k_{i,\bf{l}_i,3}(\cdot)$ we get from \eqref{feedback:ki3:plan} and \eqref{set:W} that
\begin{equation} \label{ki3:bound:plan}
|k_{i,\bf{l}_i,3}(w_i)| = |\lambda(i) w_i|\le \lambda(i) v_{\max}(i),\forall w_i\in W.
\end{equation}

\noindent \textit{STEP 4: Verification of \eqref{controler:consistency}.} In this step we exploit the bounds obtained in Step 2 in order to show \eqref{controler:consistency} for any $d_{j_1},\ldots,d_{j_{N_i}}$ satisfying \eqref{disturbance:bounds}. By taking into account \eqref{ki1:bound}, \eqref{ki2:bound} and \eqref{ki3:bound:plan} we want to prove that
\begin{align}
 L_{1}(i) & \left(\mbf{\mu}(i)\frac{d_{\max}(i)}{2}+\bf{M}(i)t\right)+\frac{d_{\max}(i)}{2\delta t} \nonumber \\
 +L_{2}(i) &\left(\frac{(\delta t-t)d_{\max}(i)}{2\delta t}+\lambda(i)v_{\max}(i)t\right)+\lambda(i) v_{\max}(i)< v_{\max}(i), \forall t\in[0,\delta t].  \label{condition:dmax:vmax:plan}
\end{align}

\noindent Due to the linearity of the left hand side of \eqref{condition:dmax:vmax:plan} with respect to $t$, it suffices to verify it for $t=0$ and $t=\delta t$. For $t=0$ we obtain that
\begin{align*}
& L_{1}(i)\mbf{\mu}(i)\frac{d_{\max}(i)}{2}+\frac{d_{\max}(i)}{2\delta t}+L_{2}(i)\frac{d_{\max}(i)}{2}+\lambda(i) v_{\max}(i)< v_{\max}(i) \iff \\
& L_{1}(i)\mbf{\mu}(i)\delta t d_{\max}(i)+d_{\max}(i)+L_{2}(i)\delta td_{\max}(i)< 2(1-\lambda(i))v_{\max}(i)\delta t,
\end{align*}

\noindent whose validity is guaranteed by \eqref{dmax:interval}. For the case where $t=\delta t$, we have
\begin{align*}
& L_{1}(i)\left(\mbf{\mu}(i)\frac{d_{\max}(i)}{2}+\bf{M}(i)\delta t\right)+\frac{d_{\max}(i)}{2\delta t}+L_{2}(i)\lambda(i) v_{\max}(i)\delta t+\lambda(i) v_{\max}(i)< v_{\max}(i) \iff \\
& d_{\max}(i)(1+L_{1}(i)\mbf{\mu}(i)\delta t)+2(L_{2}(i)\lambda(i) v_{\max}(i)+L_{1}(i)\bf{M}(i))\delta t^2< 2(1-\lambda(i))v_{\max}(i)\delta t,
\end{align*}

\noindent which also holds because of \eqref{dmax:interval}. Hence, we deduce that  \eqref{controler:consistency} is fulfilled. The proof of Part (i) is now complete. 

\noindent \textit{Proof of (ii).} The proof of Part (ii) is a direct consequence of \eqref{planning:condition} and Corollary~\ref{corollary:individual:post}. 
\end{proof}

\section{Conclusions and Future Work}

We have provided an online abstraction framework which guarantees the existence of symbolic models for forward complete multi-agent systems under coupled constraints. The derived abstractions provide for each agent an individual discrete model for an overapproximation of its reachable set over a finite time horizon. In addition, the composition of the individual agent models provides transitions which capture the evolution of the continuous time system over the horizon.  

Ongoing and future work directions include the decentralized computation of the  overapproximations of the agents' reachable sets and the application of the framework to specific network structures. In addition we aim at quantifying the tradeoff between the depth of the planning horizon and the depth of the required information in the network graph for the investigation of the local in time reachability properties of each agent.

\section{Acknowledgements}

This work was supported by the H2020 ERC Starting Grant BUCOPHSYS, the Knut and Alice Wallenberg Foundation, the Swedish Foundation for Strategic Research (SSF), and the Swedish Research Council
(VR).

\bibliographystyle{abbrv}
\bibliography{longtitles,online_references}

\end{document}